\documentclass[12pt]{amsart}

\usepackage{amsmath}
\usepackage{amsthm}
\usepackage{amssymb}
\usepackage{amscd}          			
\usepackage{bbm}            			
\usepackage{mathrsfs}           		
\usepackage{mathalfa}

\usepackage{verbatim}           		
\usepackage{enumitem}      		

\theoremstyle{plain}
\newtheorem{Theorem}{Theorem}

\newtheorem{theorem}[Theorem]{Theorem}
\newtheorem{proposition}[Theorem]{Proposition}

\newtheorem{corollary}[Theorem]{Corollary}

\newtheorem{lemma}[Theorem]{Lemma}

\theoremstyle{definition}
\newtheorem{Definition}[Theorem]{Definition}

\newtheorem{example}[Theorem]{Example}
\newtheorem{definition}[Theorem]{Definition}
\newtheorem{remark}[Theorem]{Remark}

\theoremstyle{remark}

\newcommand{\N}{\mathbb{N}}     
\newcommand{\R}{\mathbb{R}}     
 
\def\r{\R}

\newcommand{\calA}{\mathscr{A}}

\newcommand{\calC}{\mathscr{C}}

\newcommand{\calE}{\mathscr{E}}

\newcommand{\calK}{\mathscr{K}}

\newcommand{\calO}{\mathscr{O}}
\newcommand{\calP}{\mathscr{P}}

\DeclareMathOperator{\cl}{cl}				

\newcommand{\cs}{\calC_{s}}
\newcommand{\as}{\calA_{s}}

\newcommand{\ox}{\calO(X)}
\newcommand{\cx}{\calC(X)}
\newcommand{\kx}{\calK(X)}

\newcommand{\ocx}{\calO_{c}(X)}
\newcommand{\ccx}{\calC_{c}(X)}
\newcommand{\acx}{\calA_{c}(X)}

\newcommand{\csx}{\calC_{s}(X)}

\newcommand{\ossx}{\calO_{ss}(X)}

\newcommand{\bossx}{\calO_{ss}^{*}(X)}
\newcommand{\bcssx}{\calK_{ss}(X)}
\newcommand{\bassx}{\calA_{ss}^{*}(X)}

\newcommand{\bacx}{\calA_{c}^{*}(X)}
\newcommand{\bocx}{\calO_{c}^{*}(X)}
\newcommand{\bccx}{\calK_{c}(X)}
\newcommand{\bax}{\calA^{*}(X)}
\newcommand{\bcx}{\kx}

\newcommand{\bo}{\calO^{*}}

\newcommand{\basx}{\calA_{s}^{*}(X)}
\newcommand{\bcsx}{\calK_{s}(X)}
\newcommand{\ksx}{\calK_{s}(X)}
\newcommand{\bosx}{\calO_{s}^{*}(X)}
\newcommand{\bcox}{\calK_{0}(X)}
\newcommand{\kox}{\calK_{0}(X)}
\newcommand{\px}{\calP(X)}

\newcommand{\cox}{\calC_{0}(X)}

\renewcommand{\O}{\emptyset}

\def\sm{\setminus}
\def\cl{\overline}

\def\se{\subseteq}
\def\sc{\sqcup}
\def\bsc{\bigsqcup}
\def\bc{\bigcup}

\def\eps{\epsilon}

\def\la1{\lambda_1}
\def\la2{\lambda_2}
\def\la0{\lambda_{0}}
\def\la{\lambda}

\def\umn{\times}

\usepackage{times}

\begin{document}
\title{Solid-set functions and topological measures on locally compact spaces }
\author{S. V. Butler, University of California Santa Barbara} 
\date{August 11, 2018}
\subjclass[2010]{Primary 28C15; Secondary 28C99, 54E99, 54H99}
\keywords{topological measure, solid-set function, solid set, semi-solid set}
\maketitle

\begin{abstract}
A topological measure on a locally compact space is a set function on open and closed subsets which is finitely additive 
on the collection of open and compact sets, inner regular on open sets, and outer regular on closed sets. 
Almost all works devoted to topological measures, corresponding non-linear functionals, and their applications
deal with compact spaces. The present paper is one in a series that investigates topological measures and
corresponding non-linear functionals on locally compact spaces. Here we examine solid and semi-solid sets 
on a locally compact space. We then give a method of constructing topological measures from solid-set functions
on a locally compact, connected, locally connected space. The paper gives 
examples of finite and infinite topological measures on locally compact, non-compact spaces and 
presents an easy way to generate topological measures on
spaces  whose one-point compactification has genus 0 from existing examples of topological measures on compact spaces.
\end{abstract}

\section{Introduction}

The study of topological measures (initially called quasi-measures) began with papers by 
J. F. Aarnes \cite{Aarnes:TheFirstPaper}, \cite{Aarnes:ConstructionPaper}, and \cite{Aarnes:Pure}.
There are now many papers devoted to topological measures and corresponding non-linear functionals; 
their application to symplectic topology has been studied in numerous papers (beginning with \cite{EntovPolterovich}) 
and a monograph (\cite{PoltRosenBook}). 

To date, however, almost all these works deal with  topological  measures on compact spaces. 
In \cite{Aarnes:LC} J. F. Aarnes gives a definition of a topological measure on a locally compact space, 
presents a procedure for obtaining topological measures from solid set functions
on a locally compact, connected, locally connected space,  
and constructs some examples.  
While \cite{Aarnes:LC} contains many interesting ideas,  it is not entirely satisfactory. 
It contains incomplete proofs and sometimes asks the reader to adapt 
lengthy proofs from other papers to its subject matter. 
In addition, the approach in \cite{Aarnes:LC} makes heavy use of sets that are connected 
and co-connected (i.e. have connected complements). 
We do not think this is the right approach for the non-compact setting.
For example, using these sets one may end up constructing trivial topological measures (see Example 6.2 in \cite{Aarnes:LC}). 
Finally, the paper has never been published in a refereed mainstream journal. 

The construction technique employed by Aarnes for a compact space $X$ in \cite{Aarnes:ConstructionPaper} 
was later nicely simplified by 
D. J. Grubb, who used semi-solid sets in a compact space.
Grubb presented his elegant construction in a series of lectures in 1998, but, unfortunately, never published it. 
Influenced by ideas of Aarnes and Grubb, we have developed an approach for constructing topological measures 
on locally compact spaces. 
Instead of sets that are connected and have connected complements we use sets that are connected 
and whose complement has finitely many bounded and unbounded components. 
Our approach allows us 
to extend a solid set function (see Definition \ref{DeSSFLC})
to a topological measure on $X$ when $X$ is a locally compact, connected, 
and locally connected space; 
the restriction of a topological measure to solid sets with compact closure is a solid set function that 
uniquely determines the topological measure. 
We obtain an easy way to construct topological measures on non-compact locally compact spaces 
whose one-point compactification has genus 0. 
(See \cite{Aarnes:ConstructionPaper}, and section \ref{examplesTmLC} for more information about genus.)
Thus, we are able to produce a variety of topological measures on $\R^n$, half-spaces, punctured balls, etc..  

The paper is organized as follows. In section \ref{Prelim} we give necessary topological preliminaries. 
In section \ref{SolidSemisoid} we study the structure of solid and semi-solid sets. 
In section \ref{TM} we give a definition and basic properties of 
topological measures on locally compact spaces, and in section \ref{SSF} 
we do the same for solid-set functions. In section \ref{ExtBssKc}
on a locally compact, connected, and locally connected space
we extend a solid-set function from bounded solid sets to compact connected and bounded semi-solid sets. 
In section \ref{BCOX} the extension is done to the 
finite unions of disjoint compact connected sets, and in section \ref{ExttoTM} 
the extension produces a topological measure that is uniquely defined by 
a solid set function (see Theorem \ref{extThLC} and Theorem \ref{ExtUniq}). 
In section \ref{examplesTmLC} we give examples and present 
an easy way (Theorem \ref{tmXtoXha}) to generate topological measures on 
locally compact, connected, and locally connected spaces  whose one-point compactification has genus 0 
from existing examples of topological measures on compact spaces. 
In this paper by a component of a set we always mean a connected component. 
We denote by $\cl E$ the closure of a set $E$ 
and by $\partial E$ the boundary of $E$. 
We denote by $ \bsc$ a union of disjoint sets.

\begin{definition} \label{debddset}
A set $A \subseteq X$ is called bounded if $\cl A$ is compact. 
A set $A$ is solid if $A$ is  connected, and $X \sm A$ has only unbounded 
components.
A set $A$ is semi-solid if $A$ is connected, and $X \sm A $ has only finitely many components. 
\end{definition}

Several collections of sets will be used often.   They include:
$\ox$,  the collection of open subsets of   $X $;
$\cx$,  the collection of closed subsets of   $X $; and
$\kx$,  the collection of compact subsets of   $X $.
$\px$ is the power set of $X$.

Often we will work with open, compact or closed sets with particular properties.
We use subscripts $c, s$ or $ ss$ to indicate (open, compact, closed) sets that are, respectively,  connected, solid, or semi-solid. 
For example,
$ \ocx$   is the collection of open connected subsets of  $X$, and 
$ \bcsx$  is the collection of compact solid subsets of  $X$. 

Given any collection $\calE\subseteq\px$,
we denote by $\calE^*$ the subcollection of all bounded sets belonging to $\calE$.
For example,  
$\bax = \bcx \cup \bo(X)$ is the collection of compact and bounded open sets, and
$\bassx = \bcssx \cup \bossx $ is the collection of bounded open semi-solid and 
compact semi-solid sets. By $\kox$ we denote the collection of finite unions of disjoint compact connected sets.

\begin{definition}
A non-negative set function $ \mu$ on a family of sets that includes compact sets is called compact-finite if $  \mu(K) <\infty$ for 
each compact $K$.
A non-negative set function is called simple if it only assumes values $0$ and $1$. 
\end{definition}

We consider set functions that are not identically $\infty$.

\section{Preliminaries} \label{Prelim}

This section contains necessary topological preliminaries. 
Some results in this section are known, but sometimes we give proofs for the reader's convenience.  

\begin{remark} \label{netsSETS}
An easy application of compactness (see, for example, Corollary 3.1.5 in \cite{Engelking})  shows that
\begin{itemize}
\item[(i)]
If $K_\alpha \searrow K, K \se U,$ where $U \in \ox,\  K, K_\alpha \in \cx$, and $K$ and at least one of $K_\alpha$ are compact, 
then there exists $\alpha_0$ such that
$ K_\alpha \se U$ for all $\alpha \ge \alpha_0$.
\item[(ii)]
If $U_\alpha \nearrow U, K \se U,$ where $K \in \kx, \ U, \ U_\alpha \in \ox$ then there exists $\alpha_0$ such that
$ K \se U_\alpha$ for all $\alpha \ge \alpha_0$.
\end{itemize}
\end{remark}

\begin{remark} \label{OpenComp}
\begin{itemize}
\item[(a)]
Suppose$X$ is connected,  $ U \in \ocx$ and $F \in \ccx$ are disjoint sets. 
If $\cl U \cap F \ne \O$ then $U \sc F $ is connected.
\item[(b)]
If $X$ is locally compact and locally connected,
for each $x  \in X$ and each open set $U$ containing $x$, there is
a connected open set $V$ such that $ x \in V \se \cl V \se U $ and $\cl V$ is compact.
\item[(c)]
If $V = \bsc_{s \in S} V_s$ 
where $V$ and $V_s $ are open sets, then $\cl{ V_s} \cap V_t = \O$ for $ s \ne t$. 
In particular, if $X$ is locally connected, and $V = \bsc_{s \in S} V_s$ 
is a decomposition of an open set $V$  into connected components,  then all components $V_s$ are open, and 
$\cl{ V_s} \cap V_t = \O$ for $ s \ne t$. 
\end{itemize}
\end{remark}

\begin{lemma} \label{prelLemma}
Let $U$ be an open connected subset of a locally compact  
and locally connected
set $X$. Then for any $x, y \in U$ there is $ V_{xy}  \in \bocx$ such that
$ x,y \in V_{xy} \se \cl{ V_{xy}} \se U$.
\end{lemma}

\begin{proof}
Fix $x \in U$. Let 
$$A = \{ y \in U :  \exists V_{xy}  \in \bocx \mbox{  such that  } 
x,y \in V_{xy} \se \cl{ V_{xy}} \se U \}.  $$
Clearly, $A$ is open, since if $y \in A$ then $V_{xy} \se A$.
The set $ U \sm A$ is also open, since if $y \in U \sm A$ then 
by Remark \ref{OpenComp} there exists $V \in \bocx$ 
such that  $y \in V \se \cl V \se U$.
In fact, $V \se U \sm A$. (Otherwise, if $ z \in V \cap A$ then 
$ V_{xz} \cup V$ is a bounded open connected set with 
$x,y \in V_{xz} \cup V  \se \cl{ V_{xz}} \cup \cl V \se U$,  i.e. $y \in A$.)
Thus, $ U \sm A$ is also open. Since $x  \in A$, we must have $A = U$.
\end{proof}

We would like to note the following fact. (See, for example, \cite{Dugundji}, Chapter XI, 6.2)

\begin{lemma} \label{easyLeLC}
Let $K \subseteq U, \ K \in \bcx,  \ U \in \ox$ in a locally compact space $X$.
Then there exists a set  $V \in \bo(X)$ such that 
$$ K \se V \se \cl V \se U. $$ 
\end{lemma}

\noindent
In the spirit of this result we can say more, given connectedness.

\begin{lemma} \label{LeConLC}
Let $X$ be a locally compact, locally connected space, $K \subseteq U, \ K \in \bcx,  \ U \in \ox$.
If either $K$ or $U$ is connected there exist a set  $V \in \bocx$  and a set $C \in \bccx$
such that 
$$ K \se V \se C \se U. $$ 
One may take $C = \cl V$.
\end{lemma}

\begin{proof}
Case 1:  $ K \in \bccx$.   For each $ x \in K$ 
by Remark \ref{OpenComp} there is $ V_x \in \bocx$ such that
$ x \in V_x \se \cl{ V_x} \se U$. By compactness of $K$, we may write 
$ K \se V_{x_1} \cup \ldots \cup  V_{x_n}$.
Since both $K$ and $V_{x_i} $ are connected and $ x_i \in K \cap  V_{x_i}$, 
$\, K \cup V_{x_i}$ is connected for each $ i =1, \ldots, n$.
Hence,  
$$ V = \bigcup\limits_{i=1}^n V_{x_i} =  \bigcup\limits_{i=1}^n (K \cup V_{x_i}) $$
is a bounded open connected set for which 
$$ K \se V \se \cl V \se \bigcup_{i=1}^n \cl{ V_{x_i}} \se U. $$
Take $C = \cl V$. \\
Case 2: $ U  \in \ocx$. 
As in Case 1 we may find   $V_1, \ldots, V_n \in \bocx$
such that
$$ K \se V_1 \cup  \ldots  \cup V_n \se \cl{ V_1} \ldots \cup \cl{V_n} \se U .$$
Pick $ x_i \in V_i$ for  $i=1, \ldots, n$.  By Lemma \ref{prelLemma} choose
$W_i \in \bocx$ with $ x_1, \ x_i \in W_i \se \cl{ W_i}  \se U$ for $ i=2, \ldots, n$.
Then the set $V_1 \cup W_i \cup V_i$ is connected for each $ i =2, \ldots, n$.
Then 
$$ V = \bc_{i=1}^n V_i \cup \bc_{i=2}^n W_j = \bc_{i=2}^n (V_1 \cup W_i \cup V_i)$$
is  open connected and 
$$ K \se \bigcup_{i=1}^n V_i \se V \se \cl V \se \bigcup_{i=1}^n \cl{ V_i} \cup 
\bigcup_{i=2}^n \cl{ W_i} \se U.$$
Again, let $C = \cl V$.
\end{proof}

\begin{lemma} \label{LeCCoU}
Let $X$ be a locally compact, locally connected space. Suppose $K \se U, \ K \in \bcx, \ U \in \ox$. 
Then there exists $C \in \bcox$ such that $ K \se C \se U$.
\end{lemma}

\begin{proof}
Let $U = \bsc_{i \in I'} U_i$ be the decomposition  into connected components. Since
$X$ is locally connected, each $U_i$ is open, and by compactness of $K$ there exists 
a finite set $I \se I'$ such that $K \se \bsc_{i \in I} U_i$.  
Then $ K \cap U_i  = K \sm  \bsc\limits_{j \in I, \ j \ne i} U_j$
is a compact set. For each $i \in I$ by Lemma \ref{LeConLC} 
choose $C_i \in \bccx$ such that $K \cap U_i \se C_i \se U_i$. 
The set $C = \bsc_{i \in I} C_i$ is the desired set. 
\end{proof}

\begin{lemma} \label{CmpCmplBdA}
Let $X$ be a connected, locally connected space.
Let $ A \in \acx$ and 
let $B$ be a component of $X \sm A$. Then
\begin{itemize}
\item[(i)]
If $A$ is open then $B$ is closed and $\overline{A} \cap B  \neq \O$. 
\item[(ii)]
If $A$ is closed then $B$ is open and  $A \cap \overline{B} \neq \O.$ 
\item[(iii)]
$A \sqcup \bsc\limits_{s \in S} B_s$ is connected for 
any family $\{ B_s\}_{s \in S} $ of components of $ X \sm A$.
\item[(iv)]
 $B$  is connected and co-connected.
\end{itemize}
\end{lemma}

\begin{proof}
The proof of the first two parts is not difficult. For the third part, observe that by Remark \ref{OpenComp} $A \bsc B$ is 
connected for each component $B$ of $X \sm A$. To prove the last part,
let $X \sm A = \bsc_{s \in S}  B_s$ be a decomposition 
into connected components. For each $ t \in S$ connected component $B_t$ is also
co-connected, because
$$ X \sm B_t =A  \sc \bsc_{s \ne t}  B_s $$ 
is a connected set  by the previous part. 
\end{proof} 

\begin{lemma} \label{LeAaCompInU}
Let $X$ be a connected, locally connected space.
Let $K \in \bcx, \ K \subseteq U \in \bocx$. Then at most a finite number of connected components of 
$X \setminus$ K are not contained in $U.$
\end{lemma}

\begin{proof}
Let $ X  \sm K = \bsc_{s \in S} W_s$  be the decomposition of $ X \sm K$ 
into connected components.
Note that  each component $ W_s$ intersects $U$ since otherwise we would have 
$W_s \se X \sm U$, so $\cl{ W_s} \se X \sm U$, so $\cl{ W_s} \cap K =\O$, which contradicts 
Lemma \ref{CmpCmplBdA}.
Assume that there are infinitely many components of $X \sm K$ that are not contained in $U$.
Then we may choose components $W_i, \ i=1, 2, \ldots$,  such that $W_i \cap U \neq\O$ and 
$W_i \cap (X \sm U) \neq \O$ for each $i$. 
Connectivity of $W_i$ implies that $ W_i \cap \partial U \neq \O$ for each $i$.
Let $x_i \in W_i \cap \partial U$. By compactness of $\partial U$, let $x_0 \in \partial U$ 
be the limit point of $(x_i)$. Then $x_0  \in X \sm U \se X \sm K = \bsc_{s \in S} W_s$, 
i.e. $x_0 \in W_t $ for some $t \in S$. But then all but finitely many $x_i$ must also 
be in $W_t$, which is impossible, since $W_i \cap W_t =\O$ for $t \neq i$. 
\end{proof}

\begin{corollary} \label{CoBddComp}
Let $X$ be a connected, locally connected space.
Let $K \in \bcx$ and let $W$ be the union of bounded components of $X \sm K$. 
Then $W \in \bo(X)$.
\end{corollary}

\begin{proof}
By Lemma \ref{LeConLC} pick $V \in \bocx$ such that 
$ K \se V$. From Lemma  \ref{LeAaCompInU} it follows that $W$ is bounded. By Lemma 
\ref{CmpCmplBdA} $W$ is open.
\end{proof}

\begin{remark}  \label{ReUnbddComp}
If $A \se B, \ B \in \bax$ then $ X \sm B \se X \sm A$ and  each unbounded component 
of $X \sm B$ is contained in an unbounded component of $ X \sm A$.
\end{remark}

\begin{lemma} \label{LeUnbddComp}
Let $X$ be a connected, locally connected space.
Assume $ A \se B, \ B \in \bax$. Then each unbounded component of $X \sm B$ 
is contained in an unbounded component of $ X \sm A$ and each unbounded component
of $ X \sm A$ contains an unbounded component of $X \sm B$.
\end{lemma}

\begin{proof}
Suppose first that $ A \se K, \ K \in \bcx$. 
The first assertion is Remark \ref{ReUnbddComp}. Now let $E$ be an unbounded component of 
$X \sm A$ which contains no unbounded components of $X \sm K$. Then $E$ is contained 
in the union of $K$ and all bounded components of $X \sm K$. 
By Corollary  \ref{CoBddComp} this union is a bounded set, and so is $E$, which leads to a contradiction. Therefore, 
each unbounded component of $X \sm A$ must contain an unbounded component of $X \sm K$.

Now suppose $ A \se B, \ B \in \bax$. Choose $K \in \bcx$ such that $ A \se B \se K$. Let $E$ be an unbounded component of 
$X \sm A$. By the previous part, $ E$ contains an unbounded component $Y$ of  $ X \sm K$. But $Y \se G$ for some 
unbounded component $G$ of $ X \sm B$. Then $G \se E$. 
\end{proof}

\begin{lemma} \label{LeNoUnbdComp}
Let $X$ be locally compact, locally connected.
Let $A \in \bax$. Then the number of unbounded components of $ X \setminus A$ is finite.
\end{lemma}

\begin{proof}
Suppose first that $ A \in \bcx$. By Lemma \ref{LeConLC}
let $U \se \bocx$ be such that $ A \se U $. Then the assertion follows from
Lemma \ref{LeAaCompInU}.
Now suppose that $ A \in \bo(X)$.  Then $\cl A \in \bcx$, so the number of unbounded components
of $ X \sm \cl A$ is finite. From Lemma \ref{LeUnbddComp} 
it follows that the number of unbounded components of $X \sm A$  is also finite, 
since it does not exceed the number of unbounded components of $ X \sm \cl A$.
\end{proof}  

\begin{lemma} \label{LeCleverSet}
Let $X$ be locally compact, connected, locally connected.
Suppose $D \subseteq U$ where $ D \in \bcx, \ U \in \bo(X).$
Let $C$ be the intersection of the union of bounded components of $X \sm D$ 
with the union of bounded components of $X \sm U$. Then  $C$ is compact 
and $ U \sqcup C$ is open.
\end{lemma}

\begin{proof}
Write 
$$ X \sm D = V \sqcup W,$$ 
where $V$ is the union of bounded components of $ X \sm D$, 
and $W$ is the union of unbounded components of  $ X \sm D.$
Also write
$$ X \sm U = B \sqcup F,$$ 
where $B$ is the union of bounded components of $ X \sm U$, and 
$F$ is the union of unbounded components of $X \sm U.$
By Lemma \ref{LeNoUnbdComp} $F$ is a closed set.
Let 
$$ C = V \cap B.$$ 
Clearly, $C$ and $U$ are disjoint. To see that $U \sqcup C$ is open, 
note first that $U \sqcup B = X \sm F$ is an open set. Hence, 
$$ U \sqcup C = U \sqcup(V  \cap B) = (U \cup V) \cap (U \sqcup B)$$
is also an open set. Now we shall show that $C$ is closed,
i.e. that $X \sm C$ is open.  Note that $ F \se W$ by Remark \ref{ReUnbddComp}.  
The set $W$ is open by Lemma \ref{CmpCmplBdA}. Now 
\begin{eqnarray*}
X \sm C &=& X \sm (B \cap V) = (X \sm B) \cup (X \sm V)  \\
&=& ( U \sc F) \cup (D \sc W) = (U \cup D) \cup (F \cup W) = U \cup W
\end{eqnarray*} 
is an open set.  By Corollary \ref{CoBddComp} the set $C$ is bounded. 
\end{proof}

\section{Solid and Semi-solid sets}  \label{SolidSemisoid}

\begin{remark} \label{ReFinNoComp}
Let $X$ be locally compact, locally connected.
From Lemma \ref{LeNoUnbdComp}
it follows that a bounded set $B$ is semi-solid if and only if the number of bounded 
components of $X \setminus B$ is finite. For a  bounded solid set $A$ we have:
$$ X \sm A = \bsc_{i=1}^n  E_i $$ 
where $ n \in \N$ and $E_i$'s are unbounded connected components.
\end{remark}

\begin{lemma} \label{SolidCompoLC}
Let $X$ be locally compact, locally connected.
If $A \in \bax$ then each bounded component of $X \setminus A$ is a solid bounded set.
\end{lemma}

\begin{proof}
Let 
$$ X \setminus A = \bsc_{i \in I} B_i \sc \bsc_{j \in J} D_j $$
be the decomposition  of $X \setminus A$ into components, where $B_i, \ i \in I$ are bounded
components, and $D_j, \ j \in J$ are unbounded ones. Pick a bounded component $B_k$.
Then 
$$ X \setminus B_k = A  \sc \bsc_{i \neq k} B_i  \sc \bsc_{j \in J} D_j $$ 
Note that the set on the right hand side is connected by 
Lemma \ref{CmpCmplBdA} and unbounded.
Hence, $B_k$ is solid.
\end{proof}

A set $A \in \bacx$ may not be solid. But we may make it solid by filling in the "holes" 
by adding to $A$ all bounded components of $X \sm A$. More precisely, we have
the following result.

\begin{lemma} \label{leSolidHu}
Let $X$ be locally compact, locally connected.
For $A \in \bacx$ let 
 $\{A_i\}_{i=1}^n$be the unbounded components 
of $X \setminus A$ and $\{B_s\}_{s \in S}$ be the bounded components of $X \sm A$. 
Then the set $\tilde{A} = A \sc \bsc_{s \in S} B_s  =X \setminus \bsc_{i=1}^n A_i $ 
is solid.
\end{lemma}

\begin{proof}
The set $\tilde A $ is connected by Lemma \ref{CmpCmplBdA}. 
It is clear that $X \sm \tilde A$ has only unbounded components.
\end{proof}
  
\begin{definition} \label{solid hull}
Let $X$ be locally compact, locally connected.
For $A \in \bacx$ let $\{A_i\}_{i=1}^n$be the unbounded components 
of $X \setminus A$ and $\{B_s\}_{s \in S}$ be the bounded components of $X \sm A$. 
We say that $\tilde{A}=  A \sc \bsc_{s \in S} B_s= X \setminus \bsc_{i=1}^n A_i $ 
is a solid hull of $A$.
\end{definition}

The next lemma gives some properties of solid hulls of connected sets that are  bounded open or compact. 

\begin{lemma}\label{PrSolidHuLC}
Let $X$ be locally compact, connected, locally connected.
Let $A, B \in \bacx$.
\begin{enumerate}[label=(a\arabic*),ref=(a\arabic*)]
\item \label{part1}
If $ A \subseteq B$ then $\tilde{A} \subseteq \tilde{B}.$
\item \label{part2}
$\tilde{A}$ is a bounded solid set, $A \subseteq  \tilde{A}$, and $A$ is solid iff  $A   = \tilde{A}.$
\item \label{part3}
$\tilde{\tilde{A}} = \tilde{A}.$
\item  \label{part4}
If $A$ is open, then so is $ \tilde{A}$. If $A$ is compact, then so is $\tilde{A}.$
\item \label{part5}
If  $A, B$ are disjoint bounded connected sets, then their solid hulls $\tilde{A}, \tilde{B}$ are
either disjoint or one is properly contained in the other.
\end{enumerate}
\end{lemma}

\begin{proof}
Part \ref{part1} follows since each unbounded component of 
$ X \sm B$ is contained in an unbounded component of $  X \sm A$. 
 If $A$ is compact, choose by
Lemma \ref{LeConLC} a set $U \in \bocx$ that contains $A$.
Since $\tilde A$ is a union of $A$ and bounded components of $X \sm A$, applying 
Lemma \ref{LeAaCompInU} we see that $\tilde A$ is bounded. The rest of  
parts \ref{part2} and \ref{part3} is immediate. 
For part \ref{part4},
note that if $A$ is open (closed) then 
each of finitely many (by Lemma \ref{LeNoUnbdComp}) unbounded components of $X \sm A$ is closed (open) by
Lemma \ref{CmpCmplBdA}.
To prove part \ref{part5}, let $A, B \in \bacx$ be disjoint. 
If $A \se \tilde B$ then  $\tilde A \se \tilde B$ by parts \ref{part1} and \ref{part3}. 
To prove that the inclusion 
is proper, assume to the contrary that $\tilde A = \tilde B$. If one of the sets $A, B$ 
is open and the other is closed, this equality means that $\tilde A$ is a proper  
clopen subset of $X$, 
which contradicts the connectivity of $X$. Suppose $A$ and $B$ are both closed (both open).
Then it is easy to see that $A = E$, where $E$ is a  bounded component of $X \sm B$,
an open (closed) set. Thus, $A$ is a proper clopen subset of $X$, which
contradicts the connectivity of $X$. Therefore, $\tilde A$ is properly contained in $\tilde B$. 
Similarly, if $ B \se \tilde A$ then $ \tilde B \se \tilde A $, and the inclusion is proper. 
Suppose neither of the above discussed cases $A \se \tilde B$ or $B \se \tilde A$ occurs.  
Then by connectedness we must have:
$$ A \se G , \ \ B \se E$$
where $G$ is an unbounded component of $ X \sm B$ and $E$ is an unbounded 
component of $ X \sm A$. Then $ B \se \tilde B \se X \sm G \se X \sm A$,
i.e. $\tilde B$ is contained in a component of $ X \sm A$.
Since $\tilde B$ is connected and $ B \se E$ we must have $ \tilde B \se E  \se X \sm \tilde A$.
\end{proof}

\begin{lemma} \label{LeCsInside}
Let $X$ be locally compact, connected, locally connected.
If $K \se U, \ K \in \bcx, \ U \in \bosx$ then there exists $C \in \bcsx$ such that
$$ K \se C \se U.$$
\end{lemma}

\begin{proof}
One may take $C$ to be the solid hull of the set $\cl V$ from Lemma 
\ref{LeConLC}. Then $C \se U$  by Lemma \ref{PrSolidHuLC}. 
\end{proof}

\begin{lemma} \label{opensolid}
Let $X$ be locally compact, connected, locally connected.
Let $K \se V, \ K \in \bcsx, \ V \in \ox$. Then there exists $ W \in \bosx$ such that 
$$ K \se W \se \cl W\se V.$$
\end{lemma}

\begin{proof}
By Lemma \ref{LeConLC} we may choose $ U  \in \bocx$ such that
\begin{eqnarray} \label{V}
K \se U \se \cl U \se V.
\end{eqnarray}
Since $ K \in \bcsx$ let
$$ X \sm K = \bsc_{j=1}^n V_j$$
be the decomposition into connected components.  Each $V_j $ is an unbounded 
open connected set. Since $X\sm U \se X \sm K$, for each $j=1, \ldots, n$ let 
$E_j$ be the union of all  bounded components of $X \sm U$ contained in $V_j$, and 
let $F_j$ be the union of (finitely many by Lemma \ref{LeNoUnbdComp}) 
unbounded components of $X \sm U$ contained in $V_j$.
By Lemma \ref{CmpCmplBdA} each $F_j$ is closed.
By Lemma \ref{LeUnbddComp} each $F_j$ is non-empty. 
Then by Lemma  \ref{CmpCmplBdA}  non-empty set $F_j \cap \cl U \se V_j$ and 
$F_j \cap \cl U \in \bcx$.
Now, $E_j \se \tilde U$, so $ E_j $ is bounded.
Note that $X = K \sc \bsc_{j=1}^n V_j$, and a limit point $x$
of $E_j$ can not be in $V_i$ for $i \neq j$; and it can not be in $K$, since in this case a neighborhood $U$ of $x$ contains no points of 
$E_j$. Thus, $\cl{ E_j} \se V_j$. Then $(F_j \cap \cl U) \cup \cl{ E_j}$ is a compact set contained in $V_j$. 
By Lemma \ref{LeConLC}  there exists $ D_j \in \bccx$ such that 
\begin{eqnarray} \label{3sh}
(F_j \cap \cl U) \cup \cl{ E_j} \se D_j \se V_j. 
\end{eqnarray} 
Let 
$$ B_j = D_j \cup F_j.$$
Then $B_j$ is connected because from (\ref{3sh}) one sees that $D_j $
intersects every component comprising $F_j$. 
Thus,  each $B_j$ is an unbounded closed connected set, $B_j \cap K =\O$.
Set
 $$B= \bigcup_{j=1}^n B_j.$$
Then 
$B \cap K = \O$.
Now $ K \se X \sm B$, so let
$O$ be the connected component  of $X \sm B$ such that $ K \se O \se X \sm B$.
Since $B= \bigcup_{j=1}^n B_j \se X \sm O$, 
$B$ is contained in the union of unbounded components of $X \sm O$. 
Hence, each bounded component of $ X \sm O$ is disjoint from $B$, and
so $\tilde O \se X \sm B$. Thus
 $$ K \se O \se \tilde O \se  X \sm B \se U.$$
By (\ref{V}) we see that 
$$ K \se \tilde O \se U \se \cl U \se V$$
and we may take $W =\tilde O$.  
\end{proof} 

\begin{remark}
The closure of a solid set need not be solid. For example, 
in the infinite strip
$X = \r \times [0,1] $  the open set $ U = ((1,3) \times (0,1)) \cup ((5,7) \times(0,1)) \cup ((2,6) \times (0.25 , 0.75))$
is solid,  while its closure is not.
\end{remark}

\begin{lemma} \label{ossreg}
Let $X$ be locally compact, connected, locally connected.
Suppose $K \se W, \ K \in \bccx, \ W \in \ossx$. Then there exist $V \in \bossx$ and 
$ D \in \bcssx$ such that 
$$ K \se V \se D \se W.$$
\end{lemma}

\begin{proof}
By Lemma \ref{LeConLC} choose $U \in \bocx$ and $ C \in \bccx$ such that 
$$ K \se U \se C \se W.$$
Let $X \sm W = \bsc_{i=1}^n E_i, \ X \sm C = \bsc_{t \in T} V_t, \ X \sm U = \bsc_{s \in S} D_s$ 
be decompositions into connected components of $X \sm W, \ X \sm C,  \ X \sm U$ respectively.
Then 
$$  \bsc_{i=1}^n E_i \se  \bsc_{t \in T} V_t \se \bsc_{s \in S} D_s.$$
Let  $T_0 = \{ t \in T: \ V_t \mbox{   is unbounded   } \}$. 
Let us index by $T'$ the family of 
all bounded components of $X \sm C$ each of which contains a component of $X \sm W$. 
So $ \bsc_{i=1}^n E_i \se  \bsc_{t \in T_0}  V_t \sc   \bsc_{t \in T'} V_t$. Note that $T'$ is a finite index set.
Now let us index by $S'$ the family of 
all bounded components of $X \sm U$ each of which contains a component $V_t$ for some 
$t \in T' $. 
Note that $S'$ is a finite index set and 
$$ \bsc_{t \in T'} V_t \se \bsc_{s \in S'} D_s.$$
Consider 
$$ V = \tilde U \sm \bsc_{s \in S'} D_s.$$ 
Then $V$ is bounded.
Also, $V$ is open. By Lemma \ref{CmpCmplBdA} $V$ is connected. 
Since 
$$ X \sm V = (X \sm \tilde U) \sc\bsc_{s \in S'} D_s \se \bsc_{s \in S} D_s  = X \sm U$$
we see that $V \in \bossx$ (as the first equality indicates that $X \sm V$ has finitely 
many components), and that $U \se V$.
Now consider 
$$ D = \tilde C \sm \bsc_{t \in T' } V_t.$$
Then $D$ is compact. By Lemma \ref{CmpCmplBdA} $D$ is connected. 
We have 
$$ X \sm D = (X \sm \tilde C) \sc \bsc_{t \in T'} V_t  \se  (X \sm \tilde U) \sc \bsc_{s \in S'} D_s = X \sm V,$$
so $X \sm D$ has finitely many components, and $V \se D$. Thus, $D \in \bcssx$.
Also, 
$$ X \sm W = \bsc_{i=1}^n E_i \se 
\bsc_{t \in T_0} V_t \sc \bsc_{t \in T'} V_t  = (X \sm \tilde C) \sc \bsc_{t \in T'} V_t = X \sm D.$$
Therefore, $ D \se W$.
Then we have:
$$ K \se U \se V \se D \se W,$$
where $ V \in \bossx$ and $ D \in \bcssx$.
\end{proof}

Let $V$ be an open subset of $X$ endowed with the subspace topology.  Let $D \se V$.
By $\cl D^V$ we denote the closure of  $D$ in $V$ with the subspace topology. As before, 
$\cl D$ stands for the closure of $D$ in $X$.  

\begin{lemma} \label{76a}
Let $V \in \ox, \ D \se V$. Suppose $V$ is endowed with the subspace topology.
\begin{itemize}
\item[a)]
If $D$ is bounded in $V$ with the subspace topology then 
$\cl D^V = \cl D$ and $\cl D \se V$. 
\item[b)]
If $D$ is bounded in $X$ and $\cl D \se V$ then $D$ is bounded in $V$.
\end{itemize}
\end{lemma}

\begin{proof}
\begin{itemize}
\item[a)]
If $D$ is bounded in $V$(with the subspace topology) then  $\cl D^V$ is a compact subset of $V$, 
and  so is  a compact in $X$, hence, closed in $X$. That is, $\cl{ \cl D^V} =  \cl D^V$.
Since clearly $\cl D^V \se \cl D$ and $ D \se \cl D^V$, we have:
$$ \cl D \se  \cl{ \cl D^V} =  \cl D^V \se \cl D.$$
It follows that $ \cl D =  \cl D^V  \se V$.
\item[b)]
Since $\cl D$ is compact in $X$ it is easy to see that $\cl D^V$ is compact in $V$.
\end{itemize}
\end{proof}

\begin{remark} \label{bddInV}
Let  $V  \in O^*(X)$ be endowed with the subspace topology.  From Lemma \ref{76a} we see that
$D$ is bounded in $V$ iff $\cl D \se V$. Hence, $D$ is unbounded in $V$ iff 
$\cl D \cap (X \sm V) \neq \O$.
\end{remark}

The next two results give relations between being a solid set in a subspace of $X$ and 
being a solid set in $X$. 

\begin{lemma} \label{LeSolidInV}
Let $X$ be locally connected.
Let $C \se V, \ C \in \csx, \ V \in \ox$. Then $C \in \cs(V)$,  i.e. connected components
of $V \sm C$ are unbounded subsets of $V$.
\end{lemma}

\begin{proof}
Suppose $V \sm C = \bsc_{s \in S} V_s$ is the decomposition into connected components 
in $V$.  
Note that 
$$ X \sm C = (X \sm V) \sc (V \sm C) =(X \sm V)  \sc \bsc_{s \in S} V_s   $$
Assume that there exists $r \in S$ such that $V_r$ is bounded in $V$.  
By Lemma \ref{76a}  $\cl{ V_r} \cap (X \sm V) =\O$.
Also, by Remark \ref{OpenComp} $\cl{V_r} \cap V_s = \O$ for each $s \ne r$.
Thus, $ \cl{ V_r} \se C \sc V_r$.
Since $V_r \se X \sm C$ and $V_r$ is connected in $X$, assume  that $V_r$ 
is contained in a component $U$
of $X \sm C$.
Then $ V_r \se U \cap \cl{V_r}  \se U \cap (C \sc V_r) \se V_r$, so $ U \cap \cl{V_r}  \se V_r$.
Thus, $U =(U \cap \cl{V_r}) \sc (U \sm \cl V_r) =  V_r \sc (U \sm \cl V_r)$ is the disconnection of $U$, unless $U = V_r$.
This shows that $U=V_r$ is a component of $X \sm C$. But this is impossible, since
$V_r$ is bounded and $C$ is solid.
\end{proof}

\begin{lemma}
Let $A \se V, \ V \in \bosx$. If $A \in \as(V) $ then $A \in \basx$.
\end{lemma}

\begin{proof}
If $ A \in \as(V)$ then $A$ is connected in $X$ and bounded in $X$. Since $V \in \bosx$, we 
may write $X \sm V = \bsc_{i \in I} F_i$ where $F_i$ are unbounded connected components. 
Let $V \sm A = \bsc_{s \in S} E_s$ be the decomposition into connected components in $V$.
Each $E_s$ is unbounded in $V$, i.e.,  $\cl{ E_s} \cap (X \sm V) \ne \O$, hence, 
$\cl{ E_s} \cap F_i \ne \O$ for some $ i \in I$. Let $I' = \{ i \in I: \ F_i \cap \cl{ E_s} \ne \O 
\mbox{  for some   } E_s \}$, and for $i \in I'$ let $S_i = \{ s \in S : \ \cl{ E_s} \cap F_i \ne \O \}$. 
For $i \in I'$ the set $ F_i \cup \bsc_{s \in S_i} E_s$ is unbounded and connected.
Since 
$$ X \sm A = (X \sm V) \sc (V \sm A) = 
\bsc_{i \in I'} (F_i \cup \bsc_{s \in S_i} E_s)  \sc \bsc_{i \in I \sm I'} F_i $$
is a  disjoint union of unbounded connected sets, the proof is complete.
\end{proof}

Now we shall take a closer look at the structure of  an open solid or semi-solid set that contains a
closed solid or closed connected set.

\begin{lemma} \label{LeDecompV}
Let $X$ be locally compact, connected, locally connected.
Let $ C \se V,  \ C \in \ksx$. 
\begin{enumerate}[label=(\roman*),ref=(\roman*)]
\item
Suppose  $V \in \bosx$. If $V \sm C$ is connected then 
$$V = C \sc W  \mbox{     where     }  W \in \bossx.$$
If  $ V \sm C$ is disconnected then 
$$ V = C \sc \bsc_{i=1}^n V_i \mbox{     where     } V_i \in \bosx, \ i=1, \ldots, n. $$
\item
Suppose $ V \in \bossx$.  Then
$$ V = C \sc \bsc_{i=1}^n V_i \mbox{     where     } V_i \in \bossx, \ i=1, \ldots, n. $$
\end{enumerate}
\end{lemma}
 
\begin{proof} 
\begin{enumerate}[label=(\roman*),ref=(\roman*)]
\item
Suppose  $V \in \bosx$ and let 
$$X \sm V = \bsc_{s \in S} F_s$$
be the decomposition 
into connected components, so $S$ is a finite index set and  each $F_s$ is unbounded.
If $V \sm C$ is connected then taking $W = V \sm C$ we see that
$$X \sm W =  X \sm (V \sm C) = (X \sm V) \sc C =C \sc \bsc_{s \in S} F_s $$
has finitely many components, i.e. $W \in \bossx$. \\
Now assume that $V \sm C$ is not connected.  
By Lemma \ref{LeSolidInV}  and Remark \ref{bddInV} $C \in \cs(V)$ and is bounded in $V$.
The set $ V \sm C$ is also disconnected in $V$, so using Remark \ref{ReFinNoComp} let
$$V \sm C = \bsc_{i=1}^n V_i, \ n \ge 2 $$ 
be the decomposition into connected (unbounded in $V$)
components in $V$. Each $V_i$ is connected in $X$.
To show that each $V_i \in \bosx$ 
we only need to check that the components of $X \sm V_i$ are unbounded.
For simplicity, we shall show it for $V_1$. For $2 \le j \le n$ 
by  Lemma \ref{LeSolidInV} and Remark \ref{bddInV}
$\cl{ V_j}$ intersects $X\sm V$, hence, intersects some $F_s$. 
Let  $S_1 = \{ s\in S:  F_s \cap \cl{ V_j} \neq \O \mbox{  for some  }  2 \le j \le n \} $.
By Remark \ref{OpenComp} and Lemma \ref{CmpCmplBdA}
the set $(\bsc_{s \in S_1} F_s \sc C \sc \bsc_{j=2}^n V_j )$ is connected. 
It is also unbounded.
Now 
\begin{eqnarray*}
X \sm V_1 &=&  (X \sm V) \sc (V \sm V_1)  \\
&=& \bsc_{s \in S} F_s \sc C \sc \bsc_{j=2}^n V_j \\
&=& (\bsc_{s \in S_1} F_s \sc C \sc \bsc_{j=2}^n V_j ) \sc \bsc_{s \in S \sm S_1} F_s
\end{eqnarray*} 
Since $X \sm V_1$ is the disjoint union of connected unbounded sets, it follows that
$V_1$ is solid.
\item
Suppose $V \in \bossx$ and let $\bsc_{j=1}^k F_j$ be the 
components of $X \sm V$. 
By Lemma \ref{LeSolidInV}  and Remark \ref{bddInV} $C \in \cs(V)$ and is bounded in $V$.
Let
$$V \sm C = \bsc_{i=1}^n V_i, \ n \ge 1 $$
be the decomposition into connected components in $V$ according to 
Remark \ref{ReFinNoComp}.  
Each $V_i$ is connected in $X$, and 
to show that each $V_i  \in \bossx$ we only need to check that $X \sm V_i$ has finitely
many components. For simplicity, we shall show it for $V_1$. We have:
$$ X \sm V_1 = (X \sm V) \sc (V \sm V_1)  = \bsc_{j=1}^k F_j \sc C \sc \bsc_{i\ne 1} V_i  .$$
Since $X \sm V_1$ is a finite disjoint union of connected sets, the number of components
of $ X \sm V_1$ is finite, so $V_1 \in \bossx$.
\end{enumerate} 
\end{proof} 

\begin{lemma} \label{LeDecompU}
Let $X$ be locally compact, connected, locally connected.
Suppose $ C \se U, \ \  C \in \bccx, \ \ U \in \bosx$. 
If $\  U \sm \tilde C$ is disconnected then
$$ U = C \sc \bsc_{s \in S} V_s, \ \ \ V_s \in \bosx.$$
If $ \  U \sm \tilde C$ is connected then
$$ U = C \sc \bsc_{s \in S} V_s \sc W,  \ \ \ V_s \in \bosx, \ W \in \bossx.$$
\end{lemma}

\begin{proof}
Note first that $\tilde C \in \bcsx$ and $\tilde C \se U$ by Lemma \ref{PrSolidHuLC}.  
Assume that $U \sm \tilde C$ is disconnected.  By Lemma \ref{LeDecompV} we may write
$ U = \tilde C \sc \bsc_{i=1}^n U_i, \ \ U_i \in \bosx$.  But
$\tilde C = C \sc \bsc_{\alpha} V_{\alpha}$, where $V_{\alpha}$ are bounded components 
of $ X \sm C$, so  by Lemma  \ref{SolidCompoLC} each $V_{\alpha} \in \bosx$. After
reindexing, one may write $$ U = C \sc \bsc_{s \in S} V_s, \ \ \ V_s \in \bosx.$$ 
The proof for the case  when 
$U \sm \tilde C$ is connected follows  similarly from  Lemma \ref{LeDecompV}. 
\end{proof}

\begin{lemma} \label{finiteT}
Let $X$ be locally compact, connected, locally connected.
Suppose that 
\[ V = \bsc_{j=1}^m C_j \sc \bsc_{t \in T} U_t \]
where $V \in \bossx, \ C_j  \in \bcsx, \ U_t \in \bocx$. Then $T$ is finite.
\end{lemma}

\begin{proof} 
The proof is by induction on $m$.  Let $m=1$. Using Lemma \ref{LeDecompV} we have
\[ V \sm C_1 = \bsc_{i=1}^n V_i = \bsc_{t \in T} U_t.\]
Since sets $V_i$ and $U_t$ are connected, $T$ must be finite.  Now let 
$V = \bsc_{j=1}^m C_j \sc \bsc_{t \in T} U_t $ and assume that the result holds for any 
bounded open semi-solid set which contains less than $m$ compact solid sets. 
Using Lemma \ref{LeDecompV} we see that
\[ V =  C_1 \sc \bsc_{i=1}^n V_i =C_1 \sc \bsc_{j=2}^m C_j \sc \bsc_{t \in T} U_t,\]
where $V_i \in \bossx$.
All involved sets are connected, so each set $V_i$ is the disjoint union of sets from the 
collection $\{ C_2, \ldots, C_m, U_t,  t \in T \}$. By the induction hypothesis each $V_i$ 
contains finitely many sets, and it follows that $T$ is finite.  
\end{proof}

\begin{lemma} \label{finiteSP}
Let $X$ be locally compact, connected, locally connected.
If $A = \bsc_{t \in T} A_t, \ \ A , A_t \in \basx$  with at most finitely many
$A_t \in \bcsx$ then $T$ is finite. 
\end{lemma}

\begin{proof}
Assume first that $A \in \bosx$. If the cardinality $|T| > 1$ then there must be  a compact solid set 
among $A_t$, and the result follows from Lemma \ref{finiteT}. Assume now that $A \in \bcsx$
and  write  
\[ A = \bsc_{j=1}^m C_j \sc \bsc_{t \in T} U_t,  \]
where $C_j \in \bcsx, \ U_t \in \bosx$. We need to show that $T$ is finite.
By Lemma \ref{opensolid} choose $ V \in \bosx$ such that $A \se V$. Then 
from Lemma \ref{LeDecompV} we may write $V \sm A =  \bsc_{i=1}^n V_i$, where 
$ V_i \in \bossx$. Then 
\[ V =\bsc_{j=1}^m C_j \sc \bsc_{t \in T} U_t  \sc  \bsc_{i=1}^n V_i ,\]
and by Lemma \ref{finiteT} $T$ is finite.
\end{proof}

\begin{remark}
Lemma \ref{LeNoUnbdComp}, Lemma \ref{PrSolidHuLC}, Lemma \ref{opensolid},  and Lemma \ref{LeSolidInV}    
are close to Lemmas 3.5, 3.6, 3.8, 3.9, and 4.2 in \cite{Aarnes:LC}. 
Lemma \ref{LeCleverSet} is related to a part in the proof of Lemma 5.9 in \cite{Aarnes:LC}. 
The case " $V \sm C$ is disconnected" in the first part of Lemma \ref{LeDecompV}  is Lemma 4.3  in \cite{Aarnes:LC}, and 
Lemma \ref{finiteSP} is an expanded (to compact sets as well) version of Lemma 4.4 in \cite{Aarnes:LC}. 
In all instances our proofs are modified, expanded, or different, compared to the proofs in \cite{Aarnes:LC}.  
\end{remark}
 
\section{Definition and basic properties of topological measures on locally compact spaces} \label{TM}

\begin{Definition}\label{DeTMLC}
A topological measure on $X$ is a set function
$\mu:  \cx \cup \ox \to [0,\infty]$ satisfying the following conditions:
\begin{enumerate}[label=(TM\arabic*),ref=(TM\arabic*)]
\item \label{TM1} 
if $A,B, A \sc B \in \kx \cup \ox $ then
$
\mu(A\sqcup B)=\mu(A)+\mu(B);
$
\item \label{TM2}  
$
\mu(U)=\sup\{\mu(K):K \in \bcx, \  K \se U\}
$ for $U\in\ox$;
\item \label{TM3}
$
\mu(F)=\inf\{\mu(U):U \in \ox, \ F \se U\}
$ for  $F \in \cx$.
\end{enumerate}
\end{Definition} 

\begin{remark} 
It is important that in Definition \ref{DeTMLC} condition \ref{TM1} holds for sets from 
$\kx \cup \ox$. In fact, \ref{TM1} fails on $\cx \cup \ox$. See Example \ref{puncdisk} or 
Example \ref{linetm} below. 
\end{remark}

The following result gives some immediate properties of topological measures 
on locally compact spaces. 

\begin{lemma} \label{propTMLC}
The following is true for a topological measure:
\begin{enumerate}[label=(t\arabic*),ref=(t\arabic*)]  
\item \label{l1}
$\mu$ is monotone, i.e. if $ A \se B, \ A, B \in  \cx \cup \ox$ then $\mu(A) \le \mu(B)$. 
\item \label{smooth}
If an increasing net $U_s \nearrow U$, where $U_s, U \in \ox$ then
 $\mu(U_s) \nearrow \mu(U)$.
In particular, $\mu$ is additive on $\ox$. 
\item \label{l5}
$\mu( \O) = 0$.
\item \label{kl}
If $V \sc K \se U$, where $U , V \in \ox, \ K \in \bcx$ then $\mu(V) + \mu(K) \le \mu(U).$
\item \label{l2}
If $\mu$ is compact-finite then $\mu(A) < \infty$ for each $A \in \bax$.
$\mu$ is finite (i.e. $\mu(X) < \infty$) iff $\mu$ is real-valued.
\item \label{CoRegulLC}
If $X$ is locally compact, locally connected then for any $U \in \ox$ 
$$\mu(U)=\sup \{ \mu(C): \ C \in \kox, \ C\se U\}.$$
\item \label{l8}
If $X$ is connected then 
$$ \mu(X)  = \sup\{ \mu(K) :  \ K \in \bccx \}  .$$
If $X$ is locally compact, connected, locally connected then also 
 $$ \mu(X)   = \sup\{ \mu(K) :  \ K \in \bcsx \}  .$$
\end{enumerate}
\end{lemma}

\begin{proof}
\begin{enumerate}[label=(t\arabic*),ref=(t\arabic*)]  
\item
The monotonicity  is immediate from Definition \ref{DeTMLC}  if 
sets $A$ and $B$ are both open or both closed. It is also easy to show the monotonicity 
in the case when one of the sets is open and the other one is closed. 
\item
Suppose $U_s \nearrow U, U_s, U \in \ox$.  
Let compact $K \se U.$   By Remark \ref{netsSETS}, there is $t \in S$ such that $K \se U_s$ for all $s \ge t$.
Then $\mu(K) \le \mu(U_s) \le \mu(U)$ for all $s \ge t$, and we see from the inner regularity
(whether $ \mu(U) < \infty$ or $ \mu(U) =\infty$)  that $\mu(U_s) \nearrow \mu(U)$.
\item 
Easy to see since $ \mu$ is not identically $ \infty$. 
\item 
Easy to see from part \ref{TM2} of Definition \ref{DeTMLC}.
\item
If $U$ is an open bounded set then $ \mu(U) \le \mu(\cl U) < \infty$. The second statement is obvious.
\item 
By Lemma \ref{LeCCoU} for arbitrary $K \se U, \ K \in \cx, \ U  \in \ox$ there is 
$C \in \cox$ with $ K \se C \se U$. By monotonicity $ \mu(K) \le \mu(C) \le \mu(U).$
Then 
\begin{eqnarray*}
\mu(U) &=& \sup\{ \mu(K): \ K \in \cx, \ K \se U\} \\
&\le& \sup\{ \mu(C): \ C \in \cox, \ K \se C \se U\}  \le \mu(U)
\end{eqnarray*}
\item
Follows from Lemma  \ref{LeConLC} and Lemma \ref{PrSolidHuLC}.
\end{enumerate}
\end{proof}

\begin{proposition} \label{PrFinAddLC}
Let $X$ be locally compact.
A set function $\mu: \ox \cup \cx \rightarrow [0,\infty] $ satisfying \ref{TM2} and \ref{TM3} of 
Definition \ref{DeTMLC}  also satisfies \ref{TM1} if  
the following conditions hold:
\begin{enumerate}[label=(c\arabic*),ref=(c\arabic*)]
\item \label{usl1}
$\mu(U \sqcup V ) = \mu(U) + \mu(V) $ for any disjoint open sets $U,V$
\item \label{usl2}
$\mu(U) = \mu(K) + \mu(U \setminus K) $ whenever $K \subseteq U, \ K \in \bcx, \ U \in  \ox.$
\end{enumerate}
\end{proposition}

\begin{proof}
Our proof is an expanded version of the proof of Proposition 2.2 in \cite{Alf:ReprTh} where the result first appeared for 
compact-finite topological measures. 
Suppose that $\mu$ is a set function satisfying \ref{TM2}, \ref{TM3} as well as
conditions \ref{usl1} and \ref{usl2}. We need to show that $\mu$ satisfies \ref{TM1}. 
$X$ is completely regular, so it is evident from
\ref{usl1} and \ref{TM3} that $\mu$ is finitely additive on $\ox$ and on $\bcx$. Hence, we only 
need to check \ref{TM1} in the situation when $ A \in \bcx, \ B \in \ox$, and $ A \sc B$ is either 
compact or open. If  $ A \sc B$ is open then using condition \ref{usl2} we get:
$$ \mu(A \sc B) = \mu((A \sc B) \sm A) + \mu(A) = \mu(B) + \mu(A) .$$
Now suppose $A \sc B \in \bcx$. Note that \ref{TM3} implies monotonicity of $\mu$ on $\bcx$.
Let $ C \in \bcx, \ C \se B$. Then finite additivity and monotonicity of $\mu$ on $\bcx$ gives:
$$ \mu(A)  + \mu(C) = \mu(A \sc C) \le  \mu(A \sc B).$$
By \ref{TM2} 
$$ \mu(A) + \mu(B) \le \mu(A \sc B).$$
Now we will show the opposite inequality. It is obvious if $\mu(A) = \infty$, so let $\mu(A) < \infty$, and  
for $ \eps >0$ pick $ U \in \ox$ such that 
$A \se U$ and $\mu(U) < \mu(A) + \eps$. Then compact set $A \sc B$ is contained in 
the open set $B \cup U$. Also, the compact set $ (A \sc B) \sm U = B \sm U$ is contained in 
$ B \cup U$, and $ (B \cup U) \sm (B \sm U)  = U$. Applying \ref{TM2} and then condition
\ref{usl2} we see that
\begin{eqnarray*}
\mu(A \sc B) &\le& \mu(B \cup U) = \mu((B \cup U) \sm (B\sm U)) + \mu(B \sm U) \\
&=& \mu(U) + \mu(B \sm U)  \le \mu(U) + \mu(B)  \\
&\le& \mu(A)  + \mu(B) + \eps
\end{eqnarray*}
Thus, 
$$ \mu(A \sc B)  \le \mu(A) + \mu(B) .$$
This finishes the proof.
\end{proof}  

\begin{remark} \label{dualwrong}
The condition \ref{usl2} of  Proposition \ref{PrFinAddLC},  
$ \mu(U) = \mu(K) + \mu(U \sm K) $
for $U$ open and $K$ compact,  is a very useful one. 
Of course, any topological measure satisfies this condition. It is interesting to note that a 
similar condition regarding a bounded open subset of a closed set fails for topological measures, i.e. 
\[ \mu(F) = \mu(U) + \mu(F \sm U) \]
where $F$ is closed and $U$ is open bounded, 
in general is not true, as Example \ref{linetm} below shows. 
\end{remark}

\section{Solid set functions} \label{SSF}

Our goal now is to extend a set function defined on a smaller collection of subsets of $X$ than
$\ox \cup \cx$ to a topological measure on $X$. One such convenient collection is the collection of solid
bounded open and solid compact sets, and the corresponding set function is a solid set function.  

\begin{definition} \label{DeSSFLC}
A function $ \la: \basx \rightarrow [0, \infty) $ is a solid set function on $X$ if
\begin{enumerate}[label=(s\arabic*),ref=(s\arabic*)]
\item \label{superadd}
whenever  $\bsc\limits_{i=1}^n C_i \subseteq C,  \ \  C, C_i \in \bcsx$, we have 
$ \sum\limits_{i=1}^n \la(C_i) \le \la(C)$;
\item \label{regul}
$ \la(U) = \sup \{ \la(K): \ K \subseteq U , \ K \in \bcsx \}$ for $U \in \bosx$; 
\item \label{regulo}
$ \la(K) = \inf \{ \la(U) : \  K \subseteq U, \ U \in \bosx \}$ for $ K  \in \bcsx$; 
\item  \label{solidparti}
if $A = \bsc_{i=1}^n A_i, \ \ A , A_i  \in \basx$ then 
$ \la(A) = \sum\limits_{i=1}^n \la (A_i)$.
\end{enumerate}
\end{definition}

\begin{lemma} \label{PrPropSsfLC}
Let $X$ be locally compact, connected,  locally connected.
Suppose $\la$ is a solid set function on $X$. Then
\begin{itemize}
\item[(i)]
$\la(\O) = 0$ 
\item[(ii)]
if $ \bsc_{s \in S} A_s \subseteq A, $  where $A_s, A \in \basx$, then 
$\sum_{s \in S } \la(A_s)\le  \la(A)$
\end{itemize}
\end{lemma}

\begin{proof} 
From Definition \ref{DeSSFLC} we see that $\la(\O) = 0$.
Now let  $ \bsc_{s \in S} A_s \subseteq A, $  where $A_s, A \in \basx$.  Since 
$\sum_{s \in S } \la(A_s) = 
\sup \{ \sum_{s \in S'} \la(A_s) : \ S' \se S,  \ S' \mbox{  is finite } \}$, 
it is enough to assume that $S$ is finite. By regularity  
in Definition \ref{DeSSFLC} we may take all sets $A_s$ to be disjoint compact solid.
If also $A \in \bcsx$,  the assertion is just part \ref{superadd} of Definition \ref{DeSSFLC}.
If $A \in \bosx$ then there exists $ C \in \bcsx$ such that $\bsc_{s \in S } A_s \se C \se A$ by
Lemma \ref{LeCsInside}. 
Now the assertion follows 
from parts  \ref{superadd} and \ref{regul}  of Definition \ref{DeSSFLC}.
\end{proof}

\section{Extension to $\bassx \cup \bccx$} \label{ExtBssKc}

We start with a solid set function $ \la: \basx \rightarrow [0, \infty)$ 
on a locally compact, connected, locally connected space $X$.
Our goal is to extend
$\la$ to a topological measure on $X$. We shall do this in steps, each time extending 
the current set function to a new set function defined on a larger collection of sets.

\begin{definition}  \label{la1LC}
Let $X$ be  locally compact, connected, locally connected.
For $A \in\bassx \cup \bccx$ define 
$$  \la_1(A) = \la(\tilde{A}) - \sum_{i \in I} \la(B_i),$$
where $\{ B_i : \  i \in I\} $ is the family of bounded components of $X \setminus A$.
\end{definition}

By Lemma \ref{SolidCompoLC} each $B_i \in \basx$.
If  $A \in\bassx \cup \bccx$ 
then $\bsc_{i \in I} B_i \se \tilde A$ and  by 
Lemma \ref{PrPropSsfLC} 
$$\sum_{i \in I} \la(B_i) \le \la(\tilde A).$$ 

\begin{lemma} \label{Prla1LC}
The set function $\la_1: \bassx \cup \bccx \rightarrow [0, \infty) $ defined in
Definition \ref{la1LC} satisfies the following properties:
\begin{enumerate}[label=(\roman*),ref=(\roman*)]
\item \label{pa1}
$\la_1$ is real-valued and $\la_1 = \la$ on  $\basx.$
\item
Suppose $\bsc_{i=1}^n A_i  \sc \bsc_{s \in S} B_s  \subseteq  A$, where 
$A, A_i \in \bassx \cup \bccx$ and $B_s \in \basx$. Then 
$$ \sum_{i=1}^n \la_1(A_i) + \sum_{s \in S} \la_1(B_s) \le \la_1(A).$$
In particular, if $\bsc_{i=1}^n C_i \subseteq C$ where $C_i, C \in \bccx$ then 
$$ \sum_{i=1}^n \la_1(C_i) \le \la_1(C)$$
and  if $A \se B, \ A,B \in \bassx \cup \bccx $ then 
$$\la_1(A) \le \la_1(B).$$
\item
Suppose that $\bsc_{i=1}^n A_i  \sc \bsc_{s \in S} B_s =  A$, where 
$A, A_i \in \bassx \cup \bccx$ and $B_s \in \basx$ with at most finitely many of $B_s \in \bcsx$.
Then 
$$ \sum_{i=1}^n \la_1(A_i) + \sum_{s \in S} \la_1(B_s) = \la_1(A).$$
\end{enumerate}
\end{lemma}

\begin{proof}
\begin{enumerate}[label=(\roman*),ref=(\roman*)]
\item
Obvious from Lemma \ref{PrSolidHuLC}, Definition \ref{la1LC}, and Lemma \ref{PrPropSsfLC}.
\item
Suppose that $\bsc_{i=1}^n A_i  \sc \bsc_{s \in S} B_s  \subseteq  A$, where 
$A, A_i \in \bassx \cup \bccx$ and $B_s \in \basx$.
We may assume that $A \in \basx$, since the inequality
\begin{eqnarray} \label{number1}
\sum_{i=1}^n \la_1(A_i) + \sum_{s \in S} \la_1(B_s) \le \la_1(A)
\end{eqnarray}
is equivalent to 
\begin{eqnarray} \label{number2} 
\sum_{i=1}^n \la_1(A_i) + \sum_{s \in S} \la_1(B_s) + 
\sum_{t \in T} \la_1 (D_t) \le \la_1(\tilde{A}),
\end{eqnarray}
where $\{D_t: \  t \in T \}$ is the disjoint family of bounded components of $X \setminus A$, and
by Lemma \ref{SolidCompoLC} each $D_t \in \basx$.

The proof is by induction on $n$. 
For $n=0$ the statement is Lemma \ref{PrPropSsfLC}. 
Suppose now $n \ge 1$ and assume the result is true for any disjoint collection 
(contained in a bounded solid set)  of bounded semi-solid or compact connected sets among which 
there are less than $n$ non-solid sets. Assume now that we have $n$ disjoint sets 
$A_1, \ldots, A_n$ from the collection
$\bassx \cup \bccx$.
Consider a partial order on 
$\{ A_1, A_2, \ldots, A_n \}$ where $A_i \le A_j$ iff  $\tilde{A_i} \subseteq \tilde{A_j}$. 
(See Lemma \ref{PrSolidHuLC}.)
Let $ A_1, \ldots, A_p$ where $ p \le n$ be maximal elements in  $\{ A_1, A_2, \ldots A_n \}$
with respect to this partial order.
For a maximal element $A_k,  k \in \{ 1, \ldots, p\}$ 
define the following index sets:
$$ I_k = \{ i \in \{ p+1, \ldots, n\}: \,  A_i \mbox{ is contained in a bounded component of }  
X\setminus A_k \},  $$
$$ S_k = \{ s \in S : \ B_s \mbox{ is contained in a bounded component of }  
X\setminus A_k \}.  $$

Let $\{ E_{\alpha} \}_{\alpha \in H} $ be the disjoint family of bounded components 
of $X \sm A_k$. Then we may say that 
$$ I_k = \bsc_{\alpha \in H} I_{k,\alpha}, \ 
S_k = \bsc_{\alpha \in H} S_{k,\alpha}  $$
where 
$$ I_{k, \alpha} = \{ i \in \{ p+1, \ldots, n\} :  \ A_i \se E_\alpha \}, $$
$$ S_{k, \alpha} = \{ s \in S : \ B_s \se E_\alpha \}.  $$

The set $I_k$ and each set $I_{k, \alpha}$ has cardinality $< n$. 
The set $E_\alpha$ is a solid set according to Lemma \ref{SolidCompoLC}, and
\begin{eqnarray} \label{zv1}
\bsc_{i \in I_{k, \alpha}} A_i \sc \bsc_{s \in S_{k, \alpha}} B_s \se E_\alpha. 
\end{eqnarray}
By induction hypothesis
$$ \sum_{i \in I_{k, \alpha}} \la_1(A_i) +
\sum_{s \in S_{k, \alpha}} \la_1(B_s) \le \la_1(E_\alpha). $$
It follows that 
\begin{align*}
 \sum_{i \in I_k} \la_1(A_i) + \sum_{s \in S_k} \la_1(B_s)
&= \sum_{\alpha \in H}\left( \sum_{i \in I_{k, \alpha}} \la_1(A_i) +
\sum_{s \in S_{k, \alpha}} \la_1(B_s) \right)  \\
&\le \sum_{\alpha \in H} \la_1 (E_\alpha).
\end{align*}
Then using part \ref{pa1} and Definition \ref{la1LC} we have:
\begin{align} \label{Aktilde}
\la_1 (A_k) + \sum_{i \in I_k} \la_1(A_i) + \sum_{s \in S_k} \la_1(B_s) 
\le \la_1 (A_k) +\sum_{\alpha \in H} \la_1 (E_\alpha) =  \la_1(\tilde{A_k}).
\end{align}

Notice that $\tilde{A_1}, \ldots, \tilde{A_p}$, 
being the maximal elements,  are all disjoint by part \ref{part5} of  Lemma \ref{PrSolidHuLC}.
This also implies that the sets $I_k, \ k=1, \ldots, p$ are disjoint
(otherwise,  if $ i \in I_k$ and also $ i \in I_m, \  1\le k, m \le p$  
then $\tilde{A_k} \cap \tilde{A_m} \ne \O$).
Similarly, the sets $S_k, \ k=1, \ldots, p$ are also disjoint.
Consider the index set $$ S' = S \setminus \bsc_{k=1}^p S_k.$$
Note that  $ \{1, \ldots, n\} = \{1, \ldots, p\} \sqcup \bsc_{k=1}^p I_k $. 
Indeed, if $i \in \{1, \ldots, n\} \sm \{ 1, \ldots, p\}$ we must have $ A_i \se \tilde A_i \se \tilde A_k$ 
for some maximal element $A_k$ (where $ k \in \{1, \ldots, p\}$), and  since $A_i$ and $A_k$ are 
disjoint, $A_i$ must be contained in a bounded component of  $A_k$, i.e. $ i \in I_k$.
Now we have:
\begin{eqnarray*}
\sum_{i=1}^n \la_1(A_i)  &+& \sum_{s \in S} \la(B_s)  \\
&=& \sum_{k=1}^p \left( \la_1(A_k) + \sum_{i \in I_k} \la_1(A_i)  + 
 \sum_{s \in S_k} \la(B_s) \right) + \sum_{s \in S'} \la(B_s) \\
&\le& \sum_{k=1}^p \la(\tilde{A_k}) + \sum_{s \in S'} \la(B_s) \\
&\le& \la(\tilde{A})
\end{eqnarray*}
The first inequality is by formula (\ref{Aktilde}), and 
for the last inequality we applied Lemma \ref{PrPropSsfLC}, since 
$ \{\tilde{A_k}\}_{k=1}^p \bsc \{B_s\}_{ s \in S'}$ is
a collection of disjoint solid sets contained in the solid set $A$.
\item 
The proof is almost identical to the proof of the previous part, and we keep the same 
notations. Again, we may assume that  $A \in \basx$, 
since the inequalities (\ref{number1}) and (\ref{number2}) become equalities.
The proof is by induction on $n$, and the case $n=0$ is given by  
Lemma \ref{finiteSP} and part \ref{solidparti} of 
Definition \ref{DeSSFLC}. The inequalities in the induction step become equalities
once one observes that (\ref{zv1}) above becomes
$\bsc_{i \in I_{k, \alpha}} A_i \sc \bsc_{s \in S_{k , \alpha}} B_s =E_{\alpha}$
(note that $\tilde A_k \se A$, so $E_{\alpha} \se A$).
Since $\bsc_{k=1}^p \tilde{A_k} \sc \bsc_{s \in S'} B_s = A$,
the last inequality in the proof of the previous part becomes an equality  
by Lemma \ref{finiteSP}  and  part \ref{solidparti} of 
Definition \ref{DeSSFLC}. 
\end{enumerate}
\end{proof} 

\section{Extension to $\bcox$} \label{BCOX}

Our goal now is to extend the set function $\la_1$ to a set function $\la_2$ 
defined on $\bcox.$  Recall that $ K \in \bcox$ if $ K = \bsc_{i=1}^n K_i$ where $ n \in \N$ and  
$K_i \in \bccx$ for $ i=1, \ldots, n.$

\begin{definition}  \label{la2LC}
For $K = \bsc\limits_{i=1}^n K_i,$ where $ K_i \in \bccx$, let 
$$ \la_2(K ) = \sum_{i=1}^n \la_1(K_i). $$ 
\end{definition}

\begin{lemma} \label{Lela2LC}
The set function $\la_2$ from Definition \ref{la2LC}  satisfies the following properties:
\begin{itemize}
\item[(i)]
$\la_2$ is real-valued, $\la_2 = \la_1 $ on $\bccx$ and $\la_2 = \la$ on $ \bcsx$. 
\item[(ii)]
$\la_2$ is finitely additive on $\bcox$
\item[(iii)]
$\la_2$ is monotone on $\bcox.$
\end{itemize}
\end{lemma}

\begin{proof} 
The first part easily follows from the definition of $\la_2$ and 
Lemma \ref{Prla1LC}. The second part is obvious. To prove the third one, 
let $ C \subseteq K, $ where 
$C, \ K \in \bcox$. Write $ C= \bsc\limits_{i=1}^n C_i, \ K = \bsc\limits_{j=1}^m K_j,$ where 
the sets $C_i (i=1, \ldots, n)$ and $ K_j (j =1, \ldots, m)$ are compact connected.
By connectivity, each $C_i$ is contained in one of the sets  $K_j.$ Consider index sets 
$ I_j  = \{ i : \ C_i \subseteq K_j \} $ 
for each $ j = 1, \ldots, m.$  By Lemma \ref{Prla1LC}  we have
$\sum_{i \in I_j} \la_1(C_i) \le \la_1(K_j).$
Then
\begin{eqnarray*}
\la_2(C) & = & \sum_{i=1}^n \la_1(C_i)  = 
\sum_{j=1}^m \sum_{i \in I_j}  \la_1(C_i) 
\le \sum_{j=1}^m \la_1 (K_j) = \la_2 (K) 
\end{eqnarray*}
\end{proof}

\section{Extension to $\ox \cup \cx$} \label{ExttoTM}

We are now ready to extend the set function $\la_2$ to a set function $\mu$ 
defined on $\ox \cup \cx.$

\begin{definition} \label{muLC}
For an open set $U$  define 
$$ \mu(U)  = \sup\{ \la_2(K) : \ K \subseteq U , \ K \in \bcox \}, $$ 
and for a closed set $F$  let 
$$ \mu(F) = \inf \{ \mu(U): \  F \subseteq U, \ U \in \ox \}.$$ 
\end{definition}

Note that $ \mu$ may assume $ \infty$.

\begin{lemma} \label{PropMuLC}
The set function $\mu$ in Definition \ref{muLC} satisfies the following properties:
\begin{enumerate}[label=(p\arabic*),ref=(p\arabic*)]
\item \label{monotoneLC}
$\mu$ is monotone, i.e. if $A \se B, \, A,B \in \ox \cup \cx$ then $ \mu(A) \le \mu(B)$.
\item \label{finiteness}
$\mu(A) < \infty$ for each $ A \in \bax$, so $ \mu$ is compact-finite.
\item \label{ineqla2}
$\mu \ge \la_2$ on $\bcox.$
\item \label{CoAppr}
Let $ K \se V, K \in \bcx, \ V \in \ox$. Then for any positive $\eps$ there exists $ K_1 \in \bcox$ 
such that 
$ K \se K_1 \se V$ and $ \mu(K_1) - \mu(K) < \eps.$
\item \label{extla2}
$\mu = \la$ on $ \basx.$
\item\label{OpenFinAddLC}
$ \mu$ is finitely additive on open sets.
\item\label{CloFinAddLC}
If $G = F \sc K$, where $G, F \in \cx, \ K \in \kx$ then $\mu(G) = \mu(F) + \mu(K).$ In particular, 
$\mu$ is finitely additive on compact sets.
\item \label{AddBox}
$\mu$ is additive on $\ox$, i.e.
if $V = \bsc\limits_{i \in I} V_i$, where $ V, \ V_i \in \ox$ for all $ i \in I$,  then
$\mu(V) = \sum\limits_{i \in I} \mu(V_i). $
\item \label{superaddF}
If $G \sc V = F$ where $ G, F  \in \cx, \ V \in \ox$ then $ \mu(G) + \mu(V) \le \mu(F).$
\item \label{superaddU}
If $G \sc V \se U$ where $ G  \in \cx,  \ V,U  \in \ox$ then $ \mu(G) + \mu(V) \le \mu(U).$
\item \label{mula1}
$\mu = \la_1$ on $\bccx$ and $\mu = \la_2$ on $\bcox$.
\item \label{regularityLC}
$\mu(U) = \sup\{\mu(C): \ C \subseteq U , \ C \in \bcx \} , \ \ \ U  \in \ox.$
\end{enumerate}
\end{lemma}

\begin{proof}
\begin{enumerate}[label=(p\arabic*)]
\item
It is obvious that $\mu$ is monotone on open sets and on closed sets.
Let $ V \in \ox, F \in \cx$. The monotonicity in the case $ F \se V$ is obvious.
Suppose $ V \se F$. For any open set $U $ with $ F \se U$ we have $V \se U$, so $ \mu(V) \le \mu(U)$, 
Then taking infimum over sets $U$ we obtain $ \mu(V) \le \mu(F)$. 
\item
Let $K \in \kx$. By Lemma \ref{LeConLC}  choose $V \in \bocx$  and $C \in \bccx$ such that 
$ K \se V \se C \se U$. For any $D \in \bcox, D \se V$ by Lemma \ref{Lela2LC} we have $\la_2(D) \le \la_2(C)$, and 
$\la_2(C) < \infty$. 
By Definition \ref{muLC} $\mu(V) \le \la_2(C)$, and then $ \mu(K) \le \mu(V) \le \la_2(C) < \infty$. 
Thus, $ \mu$ is compact-finite.
If $U$ is an open bounded set then $ \mu(U) \le \mu(\cl{U}) < \infty$.
\item
Let $K \in \bcox.$ For any open set $U$ containing $K$  we have 
$\mu(U) \ge \la_2(K)$ by the definition of $\mu$. Then, 
again from the definition of $\mu$, $\mu(K)  \ge \la_2(K).$
\item
$\mu(K) < \infty$, so by Definition \ref{muLC} find $U \in \ox$ such that $ U \se V,  \, \mu(U)  - \mu(K) < \eps$. 
Let $U_1, \ldots, U_n$ be 
finitely many connected components of $U$ that cover $K$. 
By Lemma \ref{LeConLC} 
 pick $V_i \in \bocx$  such that
$K \cap U_i \se V_i \se \cl{ V_i} \se U_i$ for $ i =1, \ldots, n$. We may take 
$K_1 = \bsc_{i=1}^n \cl{ V_i}$, for $ K_1 \se V$ and 
$$ \mu(K_1) - \mu(K) < \mu(\bsc_{i=1}^n U_i) - \mu(K) \le \mu(U) - \mu(K) < \eps .$$
\item
First we shall show that $\mu= \la$ on $\bosx$. Let $U \in \bosx$, so by part \ref{finiteness} $ \mu(U) < \infty$. 
By Definition \ref{muLC}, given $\eps > 0$, choose $K \in \bcox$ such that
$K \se U$ and $ \mu(U) -\eps < \la_2(K)$.
By Lemma \ref{LeCsInside} there exists  $ C \in \bcsx$ such that $ K \se C \se U$.
Now using Lemma \ref{Lela2LC} and Definition \ref{DeSSFLC} we have:
\begin{eqnarray*}
\mu(U) - \eps &<& \la_2(K)  \le \la_2(C) \\
&\le& \sup \{ \la_2(C) : \ \ C \se U, \ C \in \bcsx\} \\
&=& \sup \{ \la(C) : \ \ C \se U, \ C \in \bcsx\} = \la(U).
\end{eqnarray*}
Hence, $\mu(U) \le \la(U)$.
For the opposite inequality, observe that  by Lemma \ref{Lela2LC} $\la = \la_2$ on $\bcsx$,  so  
by Definition \ref{DeSSFLC}
\begin{eqnarray*}
\la(U) &=&  \sup\{ \la(C): \ C \se U,  C \in \bcsx \} \\
&=& \sup\{ \la_2 (C): \ C \se U,  C \in \bcsx \}  \\
&\le& \sup\{ \la_2 (C): \  C \se U,  C \in \bcox \} = \mu(U).
\end{eqnarray*}
Therefore, $\mu(U) = \la(U)$ for any $ U \in \bosx$. 
Now we shall show that $\mu = \la$ on $\bcsx$.  From part \ref{ineqla2} above and Lemma \ref{Lela2LC}
we have $\mu \ge \la_2 = \la$ on $\bcsx$. 
Since $\mu = \la$ on $\bosx$, for $C \in \bcsx$ we have by Definition \ref{DeSSFLC} and 
Defintion \ref{muLC}:
\begin{eqnarray*}
\la(C) &=&  \inf \{ \la(U): \ \ U \in \bosx, \ C \se U \} \\
&=& \inf \{ \mu(U): \ \ U \in \bosx, \ C \se U \} \\
&\ge& \inf \{ \mu(U): \ \ U \in \ox , \ C \se U \}  = \mu(C)
\end{eqnarray*} 
Therefore, $\mu = \la $ on $\bcsx$. 
\item
Let $U_1 , U_2 \in \ox$ be disjoint. For any $C_i, C_2 \in \bcox$ 
with $ C_i \se U_i, \ i=1,2$ we have by Lemma \ref{Lela2LC} and Definition \ref{muLC}
$$ \la_2(C_1)  + \la_2(C_2) = \la_2(C_1 \sc C_2) \le \mu(U_1 \sc U_2).$$
Then by Definition  \ref{muLC} we obtain
$$ \mu(U_1) + \mu(U_2) \le \mu(U_1 \sc U_2).$$ 
For the converse inequality, note that given $ C \se U_1 \sc U_2, \  C \in \bcox$ we 
have $C_i = C \cap U_i \in \bcox, \ i =1,2$ (since each connected component of $C$ must be 
contained either  in $U_1$ or in $U_2$) and $C = C_1 \sc C_2$.  Then
$$ \la_2(C) = \la_2(C_1) + \la_2(C_2) \le \mu(U_1)  + \mu(U_2),$$ 
giving
$$ \mu(U_1 \sc U_2)  \le  \mu(U_1)  + \mu(U_2).$$
\item
Let $C_1, C_2$ be a compact and a closed set that are disjoint.
Given $ U \in \ox, \ C_1 \sc C_2 \se U$ we may find disjoint open sets $U_1, U_2$ such that
$$ U_1 \sc U_2 \se U, \ \ \ C_i \se U_i, \ \ i=1,2.$$
Then  by parts  \ref{OpenFinAddLC} and \ref{monotoneLC}
$$ \mu(C_1) + \mu(C_2) \le \mu(U_1)  + \mu(U_2) = \mu(U_1 \sc U_2) \le \mu(U), $$
so using Definition \ref{muLC} we have
$$ \mu(C_1)  + \mu(C_2) \le \mu(C_1 \sc C_2).$$
For the converse inequality, observe that for any $U_1, U_2 \in \ox$ such that
$C_i \se U_i, \ i=1,2$ one may find disjoint open sets $V_1, V_2$ with 
$C_i \se V_i \se U_i, \ i=1,2$.
Then  by parts  \ref{OpenFinAddLC} and \ref{monotoneLC}
$$ \mu(C_1 \sc C_2)  \le \mu(V_1 \sc V_2)  = \mu(V_1) + \mu(V_2)  \le \mu(U_1)  + \mu(U_2),
$$ 
which gives  by Definition \ref{muLC} 
$$  \mu(C_1 \sc C_2)  \le \mu(C_1)  + \mu(C_2).$$
\item
Let $ V = \bsc\limits_{i \in I} V_i$ with $V, V_i \in \ox(X)$ for all $ i \in I$.
By parts \ref{OpenFinAddLC} and  \ref{monotoneLC}
for any finite $I' \subseteq I$ 
$$ \sum_{i \in I'} \mu(V_i) = \mu(\bsc_{i \in I'} V_i) \le \mu(V). $$
Then $\sum_{i \in I} \mu(V_i) \le \mu(V).$
To prove the opposite inequality, first assume that $ \mu(V) < \infty$.
For $\eps >0 $ find a compact $C \in \kox$ contained in $V$ 
such that $ \mu(V) - \eps < \la_2(C).$ By compactness, $ C \subseteq \bsc\limits_{i \in I'} V_i$ 
for some finite subset $I'$ of $I$.
Then $C = \bsc\limits_{i \in I'} C_i$ where 
$C_i = C \cap V_i \se V_i$, and $C_i \in \kox$  for each $i \in I'.$ 
By Lemma \ref{Lela2LC} and part \ref{ineqla2} we have:
\begin{eqnarray*}
\mu(V)-\eps &<& \la_2(C) = \la_2(\bsc_{i \in I'} C_i) =  \sum_{i \in I'} \la_2(C_i) \le  \sum_{i \in I'} \mu(C_i) \\
& \le &  \sum_{i \in I'} \mu(V_i) \le \sum_{i \in I} \mu(V_i) 
\end{eqnarray*}
Therefore, $\mu(V) \le \sum\limits_{i \in I} \mu(V_i)$. 
This shows that  $\mu(V) = \sum\limits_{i \in I} \mu(V_i)$ when $ \mu(V) < \infty$.

Now suppose $ \mu(V) = \infty$. For $ n \in \N$ find a compact $ K \se V$ such that $ \mu(K) > n$.
Choose a finite index set $I_n \se I$ such that $ K \se \sc_{i \in I_n} V_i$. Then
$$  \sum_{i \in I} \mu(V_i) \ge  \sum_{i \in I_n} \mu(V_i) = \mu (\bsc_{i \in I_n} V_i) \ge \mu(K) > n.$$
It follows that $ \sum_{i \in I} \mu(V_i) = \infty = \mu(V)$.
\item
It is  enough to show the statement for the case $ \mu(F) < \infty$.
If $K  \se V, \  K \in \kox$ then $ G \sc K \se F$. By parts \ref{ineqla2}, \ref{CloFinAddLC} and 
\ref{monotoneLC} $\  \mu(G) + \la_2(K) \le \mu(G) + \mu(K) \le \mu(F)$.  
Then $ \mu(G) + \mu(V) \le \mu(F)$.
\item
It is  enough to show the statement for the case $ \mu(U) < \infty$.
If $K  \se V, \  K \in \kox$ then $F= G \sc K \se U$. By  parts \ref{ineqla2}, \ref{CloFinAddLC}, 
and  Definition \ref{muLC} $\mu(G) + \la_2(K) \le \mu(G) + \mu(K) = \mu(F) \le \mu(U)$. 
Then  $ \mu(G) + \mu(V) \le \mu(U)$.
\item
Let $C \in \bccx$.  According to Lemma \ref{SolidCompoLC} and 
Definition \ref{solid hull} write $\tilde C \in \bcsx$ as  
$ \tilde C = C \sc \bsc_{i \in I} U_i$
where $U_i \in \bosx$ are the bounded components of $X \sm C$.
Given $\eps>0$ choose by Definition \ref{DeSSFLC} $V \in \bosx$ such that 
$\tilde C \se V$ and $ \la(V) -\la(\tilde C) < \eps$. 
By parts  \ref{AddBox}, \ref{superaddF}, and \ref{monotoneLC}
$$ \mu(C) + \sum_{i\in I} \mu(U_i) =  \mu(C) + \mu (\bsc_{i\in I} (U_i))  \le \mu(\tilde C) \le \mu(V).$$
Then using part \ref{extla2} and Definition \ref{la1LC} we have:
\begin{align*}
\mu(C) &\le \mu(V) - \sum_{i \in I} \mu(U_i) = \la(V) - \sum_{i\in I} \la(U_i) \\
&\le \la(\tilde C)  - \sum_{i\in I} \la(U_i) + \eps = \la_1(C) + \eps
\end{align*}
Thus, $\mu(C) \le \la_1(C)$. By part \ref{ineqla2} and Lemma \ref{Lela2LC}
$\mu(C) \ge \la_2(C) =\la_1(C)$. So $\mu =\la_1$ on $\bccx$. From part \ref{CloFinAddLC} 
and Definition \ref{la2LC} we have $\mu = \la_2$ on $\bcox$. 
\item
Using part  \ref{ineqla2}
\begin{eqnarray*}
\mu(U) &=& \sup \{\la_2(C) : C \subseteq U , \ C  \in \bcox \} \\
&\le&  \sup \{\mu(C) : C \subseteq U , \ C  \in \bcox \} \\ 
&\le&  \sup \{\mu(C) : C \subseteq U , \ C  \in \bcx \} 
\end{eqnarray*}
For the converse inequality, given $ C \se U, \  U \in \ox, \ C \in \bcx$ choose 
by Lemma \ref{LeCCoU}   $K \in \bcox$ with $ C \se K \se U$. 
Then  by parts \ref{monotoneLC} and \ref{mula1} $\mu(C) \le \mu(K) = \la_2(K)$, so 
\begin{eqnarray*}
\sup \{\mu(C) : C \subseteq U , \ C  \in \bcx \}  &\le&  \sup \{\la_2(K) : K \subseteq U , \ K  \in \bcox \}   \\
&=& \mu(U).
\end{eqnarray*}
\end{enumerate}
\end{proof}

\begin{lemma}\label{BigLemmaLC} 
For the set function $\mu$ in Definition \ref{muLC}  
$$ \mu(U) = \mu(K) + \mu(U \sm K) $$
whenever $ K \se U, \ K \in \bcx, \ U \in \ox$.
\end{lemma}

\begin{proof} 
We shall prove the statement in steps.
Recall that $\mu = \la_1$ on $\bccx$ and $ \mu = \la_2$ on $\bcox$
by part \ref{mula1} of Lemma \ref{PropMuLC}.  \\
STEP 1. We shall show that
$\mu(U) = \mu(C) + \mu(U \setminus C) $ whenever
$C \subseteq U,  \ U \in  \bossx, \ C = C_1 \sc \ldots \sc C_n, \ C_j \in \bcsx.$ \\
Let $C = C_1 \sc C_2 \sc \ldots \sc C_n$, where each $C_j \in \bcsx$. 
The proof is by induction on $n$. Suppose $n=1$, i.e. $C \in \bcsx$.
By Lemma \ref{LeDecompV} 
\begin{eqnarray}
U = C \sc \bsc_{i=1}^n U_i 
\end{eqnarray}
where  each $U_i \in \bossx$.
By Lemma \ref{Prla1LC}
$$ \mu(U) = \mu(C) + \sum_{i=1}^n \mu (U_i). $$
Then 
$$ \mu(U) - \mu(C) =   \sum_{i=1} ^n \mu (U_i) = \mu( U \sm C), $$
where the last equality follows from additivity of $\mu$ on $\ox $ in
Lemma \ref{PropMuLC}.
Suppose that result holds for all $ C \se U,  \ U \in \bossx$ where
$C$ is the disjoint union of less than $n$ sets $C_j \in \bcsx$. 
Now let $C = C_1 \sc C_2 \sc \ldots \sc C_n$, where each $C_j \in \bcsx$. 
By Lemma \ref{LeDecompV}
\begin{eqnarray} \label{UCS1}
U \sm C_1 = \bsc_{i=1}^m U_i 
\end{eqnarray}
where  each $U_i \in \bossx$.
By connectivity each $C_j, \ j=2, \ldots, n$ is contained in one of the sets $U_i$.
For  $ i =1, \ldots, m$ let 
 $K_i$ be the disjoint union of those $C_j, \ j \in \{2, \ldots, n\} $ 
that are contained in $U_i$. Notice that each $K_i$ is the union of no more than
$n-1$ disjoint sets, and 
$\bsc_{i=1}^m K_i = \bsc_{j=2}^n C_j. $
By induction hypothesis, 
\begin{align} \label{inductS1}
\mu(U_i) = \mu(U_i \sm K_i) + \mu(K_i).
\end{align} 
By finite additivity of $\mu$  on compact sets 
in Lemma \ref{PropMuLC}
\begin{align} \label{unionS1}
\mu(C) = \mu(\bsc_{j=2}^n C_j) + \mu(C_1) 
= \mu(\bsc_{i=1}^m K_i) + \mu(C_1)  = \sum_{i=1}^m \mu(K_i) + \mu(C_1). 
\end{align}
Also we have
\begin{align} \label{UKS1}
U \sm C =  (U \sm C_1) \sm \bsc_{j=2}^n C_j = (\bsc_{i=1}^m U_i) \sm (\bsc_{i=1}^m K_i )
= \bsc_{i=1}^m (U_i \sm K_i ). 
\end{align}
By the first part of the induction proof
$$ \mu(U) = \mu(U \sm C_1) + \mu (C_1). $$
Using
(\ref{UCS1}), additivity of $\mu$ on $\bo(X)$ in
Lemma \ref{PropMuLC}, (\ref{inductS1}),  (\ref{unionS1}), and  (\ref{UKS1}) 
we obtain:
\begin{eqnarray*}
\mu(U) &=& \mu(U \sm C_1) + \mu(C_1) \\
&=& \mu(\bsc_{i=1}^m U_i) + \mu(C_1) \\
&=& \sum_{i=1}^m  \mu(U_i) + \mu(C_1) \\
&=& \sum_{i=1}^m  \mu(U_i \sm K_i) +\sum_{i=1}^m \mu(K_i)  + \mu(C_1) \\
&=& \sum_{i=1}^m  \mu(U_i \sm K_i) + \mu(C)  \\
&=& \mu(U\sm C) + \mu(C)
\end{eqnarray*}
STEP 2. We shall show that
$\mu(U) = \mu(C) + \mu(U \setminus C) $ whenever
$C \subseteq U, \ C \in \bcox, \ U \in  \bosx.$ \\
Let $C = C_1 \sc C_2 \sc \ldots \sc C_n$, where each $C_i \in \bccx$. 
The proof is by induction on $n$. Suppose $n=1$, i.e. $C \in \bccx$.
By Lemma \ref{LeDecompU} 
$$ U = C \sc W \sc \bsc_{s \in S} V_s $$
where $V_s \in \bosx, \ W \in \bossx$ ($W$ may be empty).
By Lemma \ref{Prla1LC}
$$ \mu(U) = \mu(C) + \sum_{s \in S} \mu (V_s) + \mu(W). $$
Then 
\[  \mu(U) - \mu(C) =   \sum_{s \in S} \mu (V_s) + \mu(W) = \mu( U \sm C), \]
where the last equality follows from additivity of $\mu$ on $\ox$ in
Lemma \ref{PropMuLC}.
Suppose that  the result holds for all $ C \se U,  \ U \in \bosx, \ C \in \bcox$ where
$C$ is the disjoint union of less than $n$ sets $C_i \in \bccx$. 
Now assume that $C = C_1 \sc \ldots \sc C_n, \ C_i \in \bccx$.
As in the proof of Lemma \ref{Prla1LC}, consider partial order on
$\{ C_1, \ldots, C_n \}$ where $C_i \le C_j$ iff $ \tilde C_i \se \tilde C_j$. 
Some parts of the argument here are as in the proof 
of  Lemma \ref{Prla1LC}. 
Let $C_1, \ldots, C_p, \ p \le n$ be maximal elements  in $\{ C_1, \ldots, C_n \}$ 
with respect to this partial order. Then $\tilde C_1, \ldots, \tilde C_p$ are disjoint. 
This implies that the family 
\[ \{W_s: \ s \in S\} = \bsc_{k=1}^p \{ \mbox{ bounded components of } 
X \sm C_k\} \]
is a disjoint family of sets.  Each $ W_s \in \bosx$ 
by Lemma \ref{SolidCompoLC} and 
$\bsc_{s \in S} W_s \in \bo(X)$, because $\bsc_{s \in S} W_s \se U$.
Let $I = \{1, \ldots, n\} \sm \{1, \ldots, p \}$.
For each $ i \in I$ $C_i$ is non-maximal element, and there exists $k \in \{ 1, \ldots, p\}$
such that  $ C_i \se \tilde C_k$. In other words,
each non-maximal set $C_i, \ i \in I$ is contained in a bounded component of 
$X \sm C_k$ for some maximal element $C_k$ (for some $ k \in \{1, \ldots, p\}$), 
that is $C_i \se W_s$ for some $s \in S$.
Let $S_1$ be a finite subset of $S$ such that for $s \in S_1$ 
the set $W_s$ contains some $C_i,  \ i \in I$. Let $S' = S \sm S_1$.
For each $s \in S_1$ let $C_s$ be the disjoint union of those sets $C_i,  i \in I$
that are contained in $W_s$. Since $|I| \le n-1$,  each $C_s$ is the union of no more 
than $n-1$ disjoint sets, and by induction hypothesis for each $ s \in S_1$
\begin{align} \label{2sh}
\mu(W_s ) = \mu( W_s \sm C_s) + \mu(C_s).
\end{align}
Note also that 
\begin{align} \label{1sh}
\bsc_{s \in S_1} C_s = \bsc_{i \in I} C_i. 
\end{align}
Then using Definition \ref{solid hull} and (\ref{1sh}) we see that:
\begin{eqnarray*}
\tilde C_1 \sc \ldots \sc \tilde C_p &=& C_1 \sc \ldots \sc C_p  \sc \bsc_{s \in S} W_s \\
&=& C_1 \sc \ldots \sc C_p \sc \bsc_{s \in S_1} W_s   \sc \bsc_{s \in S'} W_s \\
&=&  C_1 \sc \ldots \sc C_p \sc  \bsc_{s \in S_1} C_s  \sc \bsc_{s \in S_1} (W_s \sm  C_s) 
\sc \bsc_{s \in S'} W_s \\
&=& C_1 \sc \ldots \sc C_p \sc  \bsc_{i \in I} C_i  \sc \bsc_{s \in S_1} (W_s \sm  C_s) 
\sc \bsc_{s \in S'} W_s \\
&=& C_1 \sc \ldots \sc C_n \sc \bsc_{s \in S_1} (W_s \sm  C_s) 
\sc \bsc_{s \in S'} W_s 
\end{eqnarray*}
We write
\begin{eqnarray} \label{CtildeUnion}
\tilde C_1 \sc \ldots \sc \tilde C_p  =  C_1 \sc \ldots \sc C_n \sc W = C \sc W 
\end{eqnarray}
where
$$W = \bsc_{s \in S_1} (W_s \sm  C_s) \sc \bsc_{s \in S'} W_s$$ 
is an open bounded set (since $W \se \bsc_{s \in S} W_s \se U$).
Using Definition \ref{solid hull} and Definition \ref{la1LC}, 
(\ref{2sh}), (\ref{1sh}), finite additivity of $\mu$ on $\bcx$  
and additivity  of $\mu$ on $\ox$ in Lemma \ref{PropMuLC} we have:
\begin{eqnarray*}
\mu(\bsc_{k=1}^p \tilde C_k) &=& \sum_{k=1}^p \mu(C_k) + \sum_{s \in S} \mu(W_s) \\
&=&  \sum_{k=1}^p \mu(C_k) + \sum_{s \in S_1} \mu(W_s) + \sum_{s \in S'} \mu(W_s) \\
&=&  \sum_{k=1}^p \mu(C_k) + \sum_{s \in S_1}  \mu(C_s) + 
\sum_{s \in S_1} \mu(W_s \sm  C_s) + \sum_{s \in S'} \mu(W_s) \\
&=& \sum_{k=1}^p \mu(C_k) + \sum_{i \in I}  \mu(C_i) + 
\sum_{s \in S_1} \mu(W_s \sm  C_s) + \sum_{s \in S'} \mu(W_s) \\
&=& \mu(C_1 \sc \ldots \sc C_n) + \mu(W) =\mu(C) + \mu(W).
\end{eqnarray*} 
The sets $ U \sm (C \sc W) = U \sm (\tilde C_1 \sc \ldots \sc \tilde C_p)$ and $W$ are 
disjoint open bounded sets whose union is $U \sm C$. 
Now using the result of Step 1, (\ref{CtildeUnion}), just obtained equality 
$\mu(\bsc_{k=1}^p \tilde C_k) = \mu(W) + \mu(C)$,
and  additivity of $\mu$ on $\ox$ in Lemma \ref{PropMuLC} we have:
\begin{eqnarray*}
\mu(U) &=& \mu(U \sm \bsc_{k=1}^p \tilde C_k)  + \mu(\bsc_{k=1}^p \tilde C_k) \\
&=& \mu(U \sm (C \sc W )) + \mu(W) + \mu(C) \\
&=& \mu(U \sm C) + \mu(C)
\end{eqnarray*}
STEP 3. We shall show that
$\mu(U) = \mu(K) + \mu(U \setminus K) $ whenever $K \subseteq U, \ K \in \bcx, \ U \in  \bocx.$ \\
Using part \ref{regularityLC} of Lemma \ref{PropMuLC} and  Lemma \ref{LeConLC}  
choose sets $ W \in \bocx$ and $ D \in \bccx$ such that 
\begin{eqnarray} \label{WD}
 K \se W \se D \se U \mbox{   and   } \mu(U) - \mu(W) < \eps. 
\end{eqnarray}
Let $B$ be the union 
of bounded components of $ X \sm U$ and let the open set $V$ 
be the union of bounded components of $ X \sm D$. Set 
$$C = B \cap V.$$
By Lemma \ref{LeCleverSet}
$C$ is compact and $ U \sqcup C$ is open.
The solid hull $\tilde D= D  \sqcup V$. 
Then by part \ref{superaddF} of Lemma \ref{PropMuLC}
$\mu(D) + \mu(V) \le \mu(\tilde D)$. 
Note that by Lemma \ref{PrSolidHuLC}
$V \se \tilde D  \se \tilde U  = U \sc B$. 
Then
$$ V \se  U \sc (B \cap V) = U \sc C. $$
It follows that
$$ K \sc C \se D \sc V = \tilde D \se U \sc C .$$
Since $U \sc C$ is open, by Lemma \ref{opensolid}  we may find $ W' \in \bosx$ such that 
\begin{eqnarray} \label{sha} 
K \sc C \se \tilde D \se W' \se U \sc C.
\end{eqnarray}
Then 
\begin{eqnarray} \label{W'} 
W' \sm(K \sqcup C) \se U \sm K.
\end{eqnarray}
According to part \ref{CoAppr} of Lemma \ref{PropMuLC}, pick $K_1 \in \bcox$ such that 
\begin{eqnarray} \label{K1}
K \sqcup C \se K_1 \se W'  \mbox{   and    } \mu(K_1) \le \mu(K \sqcup C) + \eps.
\end{eqnarray}
By Step 2, $\mu(W') = \mu(W' \sm K_1) + \mu(K_1)$.
Now using (\ref{WD}), Definition \ref{la1LC}, 
(\ref{sha}),  (\ref{K1}), (\ref{W'}), additivity on $\ox$ and finite additivity of 
$\mu$ on $\bcx$ in Lemma \ref{PropMuLC}  we have:  
\begin{eqnarray*}
\mu(U) - \eps  &<& \mu(W) \le \mu(D)= \mu(\tilde D) - \mu(V) \\
&\le& \mu(\tilde D) - \mu(C)  \\
&\le&  \mu(W') - \mu(C)  \\
&=& \mu(W' \sm K_1) + \mu(K_1) - \mu(C) \\
&\le& \mu(W'  \sm (K \sqcup C)) + \mu(K \sqcup C) + \eps - \mu(C) \\
&\le& \mu(U \sm K)  + \mu(K) + \mu(C) -\mu(C) + \eps\\
&=& \mu(U \sm K)  + \mu(K)  + \eps
\end{eqnarray*}
It follows that $\mu(U)  \le \mu(U \sm K)  + \mu(K)$. 
The opposite inequality is part \ref{superaddU} of Lemma \ref{PropMuLC}. \\
STEP 4.  We shall show that
$\mu(U) = \mu(K) + \mu(U \setminus K) $ whenever $K \subseteq U, \ K \in \bcx, \ U \in  \bo(X).$ \\
Let $ U = \bsc_{i \in I} U_i$ be the decomposition of $U$ into connected components, and 
let $I'$ be a finite subset of $I$ such that $ K \se \bsc_{i \in I'} U_i$. For each $i \in I'$ the set 
$K_i = K \cap U_i = K \sm \bsc_{j \in I' \sm {i}} U_j \in \bcx$ and 
\begin{eqnarray} \label{Bb}
K = \bsc_{i \in I'} K_i. 
\end{eqnarray}
By Step 3  we know that 
\begin{eqnarray} \label{A} 
\mu(K_i)  + \mu(U_i \sm K_i) = \mu(U_i) \ \ \ \ \   \mbox{for  each} \ \  i \in I'.
\end{eqnarray}
Then  using (\ref{Bb}),  finite additivity of $\mu$ on $\bcx$ and additivity of $\mu$ on 
$\ox$ in Lemma \ref{PropMuLC}, and (\ref{A})  we have:
\begin{eqnarray*}
\mu(K) &+& \mu(U \sm K) = \mu(\bsc_{i \in I'}  K_i)  + \mu (U \sm \bsc_{i \in I'}  K_i) \\
&=& \sum_{i \in I'} \mu(K_i) + \sum_{i \in I'} \mu(U_i \sm K_i) + \sum_{i \in I \sm I'} \mu(U_i) \\
&=& \sum_{i \in I'} \mu(U_i) + \sum_{i \in I \sm I'} \mu(U_i) = \sum_{i \in I} \mu(U_i) = \mu(U)
\end{eqnarray*} 
STEP 5. We shall show that
$\mu(U) = \mu(K) + \mu(U \setminus K) $ whenever $K \subseteq U, \ K  \in \bcx, \ U \in  \ox.$  \\
First assume that $ \mu(U) < \infty$.
Given $\eps >0$ by Definition \ref{muLC} find $C \in \bcx $ such that $ K \se C$ and $\mu(U) - \mu(C) < \eps $.
Using  Lemma \ref{easyLeLC} find 
$ V \in \bo(X)$  such that 
$$ K \se C \se V \se U.$$ 
By Step 4 $\ \ \mu(V) = \mu(V \sm K)  + \mu(K)$.
Then using monotonicity of $\mu$ in Lemma \ref{PropMuLC} we see that
\begin{align} \label{muCVK}
 \mu(C) \le \mu(V) =  \mu(V \sm K)  + \mu(K) \le \mu(U \sm K) + \mu(K).
\end{align}
Then $ \mu(U) - \eps < \mu(C) \le \mu(U \sm K) + \mu(K)$.
Therefore, $ \mu(U)  \le \mu(U \sm K) + \mu(K)$.
The opposite inequality is part  \ref{superaddU} of Lemma \ref{PropMuLC}. 
Therefore, if $ \mu(U) < \infty$ then
$$ \mu(U)  = \mu(U \sm K) + \mu(K).$$
Now assume $ \mu(U) = \infty$.
For $n \in \N$ choose $ C \in \kx$ such that $ K \se C$ and $ \mu(C) > n$.
By Lemma \ref{easyLeLC} find $ V \in \bo(X)$  such that 
$$ K \se C \se V \se U.$$ 
Using again (\ref{muCVK}) we have: 
$$ n < \mu(C) \le \mu(V \sm K) + \mu(K),$$
i.e. $n - \mu(K) \le \mu(V \sm K) \le \mu(U \sm K)$.
Since $ \mu(K) \in \R$ by part \ref{finiteness} of Lemma \ref{PropMuLC}, 
it follows that $ \mu (U \sm K) = \infty$, and $ \mu(U \sm K) + \mu(K) = \mu(U)$.
\end{proof}

\begin{theorem} \label{extThLC}
Let $X$ be locally compact, connected, locally connected.
A solid set function on $X$ extends uniquely to a compact-finite topological measure on $X$.
\end{theorem}

\begin{proof}
Definitions \ref{la1LC}, \ref{la2LC} and \ref{muLC} extend solid set function 
$\la$ to a set function $\mu$. We would like to show that $\mu$ is a topological measure.
Definition \ref{muLC} and part \ref{regularityLC} of  Lemma \ref{PropMuLC} 
show that $\mu$ satisfies \ref{TM2}  and \ref{TM3} of definition \ref{DeTMLC}.
Proposition \ref{PrFinAddLC}, part  \ref{OpenFinAddLC} of Lemma \ref{PropMuLC}, and Lemma \ref{BigLemmaLC} 
show that $\mu$ is a topological measure.  By part \ref{finiteness} of Lemma  \ref{PropMuLC} $\mu$ is compact-finite. 

To show that the extension from a solid set function to a topological measure 
is unique suppose $\nu$ is a topological measure on $X$ such that $ \mu = \nu = \la$ on $ \basx$.
If $A \in \bccx$ then by Definition \ref{solid hull} $A = \tilde A \sm (\bsc_{s \in S} B_s)$, where $ \tilde A, B_s \in \basx$, 
so from Definition \ref{la1LC} it follows that $\mu = \nu$ on  $ \bccx$, and, hence, on $\bcox$. 
From part \ref{CoRegulLC} of Lemma \ref{propTMLC} it then follows that $\mu = \nu$ on $\ox$, so $ \mu = \nu$.
\end{proof} 

\begin{remark} \label{extsumme}
We will summarize the extension procedure for obtaining a topological measure $\mu$ from a
solid set function $\la$ on a locally compact, connected, locally connected space.
First, for a compact connected set $C$ we have:
$$ \mu(C) = \la(\tilde C) - \sum_{i=1}^n \la(B_i), $$
where $\tilde C$ is the solid hull of $C$ and $B_i$ (open solid sets) are bounded components of $X \sm C$.  
For $C \in  \bcox$, i.e. for a compact set which is the 
union of finitely many disjoint compact connected sets  $C_1, \ldots, C_n$, 
we have:
$$ \mu (C) = \sum_{i=1}^n  \mu(C_i). $$
For an open set $U$ we have: 
$$ \mu(U)  = \sup\{ \mu(K) : \ K \subseteq U , \ K \in \bcox \}, $$ 
and for a closed set $F$  let 
$$ \mu(F) = \inf \{ \mu(U): \  F \subseteq U, \ U \in \ox \}.$$ 
\end{remark}

\begin{theorem} \label{Tpart2}
If a solid set function $\la$ is extended to a topological measure $\mu$ 
then the following holds: if $\la: \basx \rightarrow \{0,1\}$ then $\mu$ also assumes only 
values $0$ and $1$; if 
$\sup \{ \la(K): \ K \in \bcsx\} = M < \infty$ then $\mu$ is finite and $ \mu(X) = M.$
\end{theorem}

\begin{proof}
Follows from Remark \ref{extsumme}, part \ref{extla2} of Lemma \ref{PropMuLC}, and part \ref{l8}  of Lemma \ref{propTMLC}. 
\end{proof}

\begin{theorem} \label{ExtUniq}
The restriction $\la$ of a compact-finite topological measure $\mu$ to $\basx$ is a solid set function, and
$\mu$ is uniquely determined by $\la$. 
\end{theorem}

\begin{proof}
Let $\la$ be the restriction of $\mu$ to $ \basx$. 
Monotonicity of a topological measure (see  Lemma \ref{propTMLC}) and \ref{TM1} of Definition  \ref{DeTMLC}
show that $\la$ satisfies conditions \ref{superadd} and \ref{solidparti}
of Definition  \ref{DeSSFLC}. 
For $ U \in \bosx$ by \ref{TM2} let $ K \in \kx$ be such that $\mu(U) - \mu(K) < \eps$ and by Lemma \ref{LeCsInside}
we may assume that $K \in \bcsx$. 
Part \ref{regul} of Definition  \ref{DeSSFLC} follows.  
Part  \ref{regulo} of Definition  \ref{DeSSFLC} follows from \ref{TM3} and Lemma \ref{opensolid}. 
Since $\mu$ is compact-finite, $\la$ is real-valued. Therefore, $\la$ is a solid set function.
\end{proof}

\begin{remark} \label{additionalPropMu}
Lemma  \ref{PrPropSsfLC}, Lemma \ref{Prla1LC},  and Lemma \ref{PropMuLC} give us 
some additional properties of topological measures. For 
example, by part \ref{CloFinAddLC} of Lemma \ref{PropMuLC}, if a closed set $F$ and  a compact $K$ are disjoint,
then $\mu(F \sc K) = \mu(F)  + \mu(K)$.
\end{remark} 

\section{Examples} \label{examplesTmLC}

When $X$ is compact, a set is called solid if it and its complement are both connected.
For a compact space $X$ we define a certain topological characteristic, genus. 
See \cite{Aarnes:ConstructionPaper} for more information about genus $g$ of the space. 
We are particularly interested in spaces with genus 0. One way to
describe the ``$g=0$'' condition is the following: if the union of two open
solid sets in $X$ is the whole space, their intersection must be connected. 
(See \cite{Grubb:IrrPart}.) Intuitively, $X$ does not have holes or loops.
In the case where $X$ is locally path connected, $g=0$ if the fundamental group $\pi_1(X)$ is finite (in particular, if $X$ is 
simply connected). Knudsen \cite{Knudsen} was able to show that if 
$H^1(X) = 0 $ then $g(X) = 0$, and in the case of CW-complexes the converse also holds.

The following two remarks for a compact space follow from results in \cite{Aarnes:ConstructionPaper}:

\begin{remark} \label{genconn}
$g(X) =0$  if and only if  $X \sm \bigsqcup\limits_{i=1}^n  C_i$ is connected
for any finite disjoint family $\{C_i\}_{i=1}^n$ of closed solid sets.  
\end{remark}

\begin{remark} \label{ReIrrPart}
If there is only one open (closed) solid set  in a solid partition of $X$ 
(i.e. a partition of $X$ into a union of disjoint sets each of which is open solid or closed solid),  
then there is  only one closed  (open) solid set in this partition.
\end{remark}

\begin{remark}
When $X$ is compact, a solid-set function on $X$ extends in a unique way to a topological measure on $X$.
For precise definitions and extension procedure see \cite{Aarnes:ConstructionPaper}.
\end{remark}

The majority of existing examples of topological measures on compact spaces are given for spaces with genus 0. 
Here is one:

\begin{example}  [Aarnes circle measure] \label{Aatm}
Let $X$ be the unit square and $B$ be the boundary of $X.$
Fix a point $p$ in $X \setminus B$.
Define $\mu $ on solid sets as follows:
$\mu (A) = 1$ if i) $B \subset A$ or
ii) $ p \in A $ and $A \cap B \ne  \O$.
Otherwise, we let $ \mu(A) = 0 $.
Then 
$ \mu $ is a solid set function and, hence,
extends to a topological measure on $X$.
Note that $\mu$ is not a point mass. 
To demonstrate that $\mu $ is 
not a measure we shall show that $\mu$ is not
subadditive. 
Let $A_1$ be a closed solid set consisting of
two adjacent sides of $B$, 
$A_2$ be a closed solid set that is the other two 
adjacent sides of $B$, and $A_3 = X \setminus B$ be
an open solid subset of $X$. Then 
$X =  A_1 \cup A_2 \cup A_3, \  \mu(X) = 1 $, but 
$ \  \mu (A_1) + \mu(A_2) + \mu(A_3) = 0$.
\end{example}

The reason  that the majority of existing examples of topological measures on compact spaces are given for the 
spaces with genus 0 is 
the following. To obtain a topological measure it is enough to define a solid-set function.
When a space has genus 0, in the definition of a solid-set function
the hardest condition to verify, the irreducible partition condition, becomes easy to verify. 
When $X$ is locally compact,  the hardest condition in Definiton  \ref{DeSSFLC} 
to verify is the condition \ref{solidparti} that deals with solid partitions. 
But, as we shall see in this section, it turns out that this condition holds trivially for spaces 
whose one-point compactification has genus $0$.

In this section we denote by $\hat X$ the  one-point compactification of $X$.

\begin{lemma} \label{hatXsoli}
Let $X$ be locally compact and $\hat X$ be its one-point compactification.
If $A \in \basx$ then $A$ is solid in $\hat X$.
\end{lemma}

\begin{proof}
Since $A $ is connected in $X$, it is also is connected in $\hat X$.
Let $X \sm A = \bsc_{i=1}^n B_i$ be the decomposition into connected components.
Each $B_i$ is an unbounded subset of $X$. 
We can write $\hat X \sm A = \bc_{i=1}^n E_i$ where each $E_i = B_i \cup \{ \infty\}$.
It is easy to see that each $E_i$ is connected in $\hat X$.  Thus,  $\hat X \sm A$ is connected,
and so  $A$ is solid in $\hat X$.
\end{proof}

\begin{lemma} \label{nosopart}
Let $X$ be a locally compact space whose one-point compactification $\hat X$ has genus 0. 
If $A \in \basx$ then any solid partition of $A$ is the set $A$ itself.
\end{lemma}

\begin{proof} 
Suppose first that $V \in \bosx$ and its solid partition is given by  
\[ V = \bsc_{i=1}^n C_i \sc \bsc_{j=1}^m U_j \]
where each $C_i \in \bcsx$ and each $U_j \in \bosx$. From Lemma \ref{hatXsoli}
it follows that $\hat X \sm V$ and each $C_i$ are closed solid sets in $\hat X$.
Since $\hat X$ has genus $0$, by Remark \ref{genconn}  
\[ \hat X \sm ((\hat X \sm V) \sc \bsc_{i=1}^n C_i)  = \bsc_{j=1}^m U_j \]
must be connected in $\hat X$. It follows that $m=1$ and we may write 
\[ V = \bsc_{i=1}^n C_i \sc  U_1. \]
Then $\{ U_1, \hat X \sm V, C_1, \ldots, C_n\} $ is a solid partition of $\hat X$, and it has 
only one open set. By Remark \ref{ReIrrPart} this solid partition also has only one 
closed set in it,  and it must be $\hat X \sm V$. So each $C_i= \O$, and the solid partition of $V$ is
$V = U_1$, i.e. the set itself.

Now suppose that $C \in \bcsx$ and its solid partition is given by  
\[ C = \bsc_{i=1}^n C_i \sc \bsc_{j=1}^m U_j \]
where each $C_i \in \bcsx$ and each $U_j \in \bosx$. 
Then $\{ \hat X \sm C, U_1, \ldots, U_m, C_1, \ldots, C_n\}$ is a solid partition of $\hat X$.
Again by Remark \ref{genconn}
\[ \hat X \sm \bsc_{i=1}^n C_i =  (\hat X \sm C) \sc U_1 \ldots \sc U_m \]
must be connected in $\hat X$. It follows that $U_j = \O$ for $j=1, \ldots, m$.
Then by connectivity of $C$ we see that the solid partition of $C$ must be the set itself.
\end{proof}

\begin{remark} \label{easyRn} 
From Lemma \ref{nosopart}
it follows that for any locally compact space whose one-point compactification has genus 0 
the last condition of Definition \ref{DeSSFLC} holds trivially. This is true, for example, 
for $X = \r^n $, half-plane in $\r^n$ with $n \ge 2$, or for a punctured ball in $\r^n$ with the relative topology.
\end{remark}

\begin{example}
Lemma \ref{nosopart} may not be true for spaces whose  one-point compactification 
has genus greater than $0$.  For example, let $X$ be an infinite strip  $\r \umn [0,1]$ without the ball 
of radius $1/4$ centered at $(-1/2, 1/2)$,
so  $\hat X$ has genus greater than $0$.   
It is easy to give an example of  a solid partition of a bounded solid set 
(say, rectangle $[0,n] \umn [0,1]$ or $(0,n) \umn [0,1]$) which consists of $n$ solid sets 
(rectangles of the type $(i, i+1) \umn [0,1]$ or  $[i, i+1] \umn [0,1]$) for any given odd $n \in \N, \, n >1$.
\end{example}

We are ready to give examples of topological measures on locally compact spaces. 

\begin{example} \label{ExDan2pt}
Let $X$ be a locally compact space whose one-point compactification has genus 0. 
Let $\la$ be a real-valued topological measure on $X$ 
(or, more generally, a real-valued deficient topological measure on $X$; 
for definition and properties of deficient topological measures on locally compact spaces see \cite{Butler:DTMLC}). 
Let $P$ be a set of two distinct points. 
For each $A  \in \basx$ let $ \nu(A) = 0$ if $\sharp A = 0$,  $ \nu(A) = \la(A) $ if $\sharp A = 1$, and 
$ \nu(A) = 2 \la(X)$ if $\sharp A = 2$,
where $ \sharp A$ is the number of points in $ A \cap P$.
We claim that $\nu$ is a solid set function. 
By Remark \ref{easyRn} we only need to check the first three conditions of 
Definition \ref{DeSSFLC}. The first one is easy to see. 
Using Lemma \ref{LeCsInside} and Lemma \ref{opensolid} it is easy to verify conditions \ref{regul} and \ref{regulo} 
of Definition \ref{DeSSFLC}. 
The solid set function $\nu$ extends to a unique finite topological measure on $X$. 
Suppose, for example, that  $ \la$ is the Lebesgue measure on $X = \r^2$, the set $P$ consists of two points
$p_1 = (0,0)$ and $p_2 = (2,0)$. Let $K_i$ be the closed ball of radius $1$ centered at $p_i$ for $i=1,2$. Then 
$K_1, K_2$ and $ C= K_1 \cup K_2$ are compact solid sets, $\nu(K_1) = \nu(K_2) = \pi, \,  \nu(C) = 4 \pi$. Since 
$\nu$ is not subadditive, $\nu$ is a topological measure that is not a measure.
\end{example}

The next two examples are adapted from Example 2.2 in \cite{Aarnes:LC} and are related to Example \ref{Aatm}.

\begin{example} \label{puncdisk}
Let $X$ be the unit disk on the plane with removed origin. $X$ 
is a locally compact  Hausdorff space with respect to the relative topology. 
Any subset of $X$ whose closure in $\r^2$ contains the origin is unbounded in $X$.
For $A \in \basx$ (since $A$ is also solid subset of the unit disk by Lemma \ref{hatXsoli})  we define 
$\mu' (A) = \mu(A)$ where $\mu$ is the solid set function on the unit disk 
from Example \ref{Aatm}.  From Remark \ref{easyRn}, Lemma \ref{LeCsInside}, Lemma \ref{opensolid} 
and the fact that $\mu$ is a solid set function
on $\hat X$  we see that $\mu'$ is a solid-set function on $X$. By Theorem \ref{extThLC}
$\mu'$ extends uniquely to a topological measure on $X$, which we also call $\mu'$.  
Note that $\mu'$ is simple. We claim that $\mu'$ 
is not a measure.  Let $U_1 = \{ z \in X:  \ Im \ z > 0\}, \ U_2 =   \{ z \in X:  \ Im \ z < 0\}$ and 
$F =  \{ z \in X:  \ Im \ z = 0\}$. Then $U_1, U_2$ 
are open (unbounded) in $X$  and $F$ is a closed (unbounded) set in $X$ consisting of 
two disjoint segments.  Note that $X= F \cup U_1 \cup U_2$. 
Using Remark \ref{extsumme} we calculate $\mu'(F) = \mu'(U_1) = \mu'(U_2) =0$.
The boundary of the disk, $C$, is a compact connected set, $X \sm C$ is unbounded in $X$, so $C \in \ksx$. 
Since $\mu'(C) = 1$, we have $\mu'(X) =1$. Thus, $\mu'$ is not subadditive, so it is not a measure. 

This example also shows that on a locally compact space
finite additivity of topological measures holds on $\kx \cup \ox$ by 
Definition \ref{DeTMLC}, but fails on $\cx \cup \ox$. This is in contrast to topological measures
on compact spaces, where finite additivity holds on $\cx \cup \ox$. 
\end{example}

\begin{example} \label{linetm}
Let $X = \r^2, \ l$ be a straight line and $p$ a point of $X$ not on the line $l$. 
For $A \in \basx$  define $\mu(A) = 1$ if $A \cap l \neq \O$ and $p \in A$; otherwise,
let $\mu(A) =0$.  Using Lemma \ref{LeCsInside} and Lemma \ref{opensolid} 
it is easy to verify the first three conditions of Definition \ref{DeSSFLC}.
From Remark \ref{easyRn} it follows that $\mu$ is a solid set function on $X$. 
By Theorem \ref{extThLC} $\mu$ extends uniquely to a topological measure on $X$, 
which we also call $\mu$.  Note that $\mu$ is simple. We claim that $\mu$ is not a measure. 
Let $F$ be the  closed half-plane determined by $l$ which does not contain $p$.  
Then using Remark \ref{extsumme} we calculate $\mu(F) = \mu(X \sm F) = 0$, 
and $\mu(X) = 1$. Failure of subadditivity shows that $\mu$ is not a measure.

The sets $F$ and $X \sm F$ are both unbounded.  Now take a bounded open disk 
$V$ around  $p$ that does not intersect $l$. Then 
\[ X = V \sc (X \sm V), \]\
where $V \in \bo(X), \ \mu(V) = \mu(X \sm V) = 0$, while $\mu(X) =1$.

This example also shows that on a locally compact space
finite additivity of topological measures holds on $\kx \cup \ox$ by 
Definition \ref{DeTMLC}, but fails on $\cx \cup \ox$. 
It fails even in the situation $X = V \sc F$, where $ V $ is a bounded open set, and $F$ is a closed set.
\end{example}

The last two examples suggest that having a topological measure on $\hat X$
helps us to get a topological measure on $X$. In fact, we have the following result.

\begin{theorem} \label{tmXtoXha}
Let $X$ be a locally compact, connected, locally connected space whose one-point compactification $\hat X$ has genus 0. 
Suppose $\nu$ is a solid set function on $\hat X$.  For $A \in \basx$  define $\mu(A) = \nu(A)$. 
Then $\mu$ is a solid set function on $X$  and, thus, extends uniquely to a topological 
measure on $X$.
\end{theorem}

\begin{proof}
Let $A \in \basx$. By Lemma \ref{hatXsoli},  $A$ is a solid set in $\hat X$. 
Using Lemma \ref{LeCsInside}, Lemma \ref{opensolid}, the fact that $\nu$ is a solid set function on $\hat X$, and that 
a bounded solid set does not contain $\infty$ it is easy to verify the first three conditions of Definition \ref{DeSSFLC}.
By Remark \ref{easyRn} $\mu$ is a solid set function on $X$. 
\end{proof}

Theorem \ref{tmXtoXha} allows us to obtain a large variety of topological measures on a locally compact space from
examples of topological measures on compact spaces. 

\begin{example} \label{nvssf}
Let $X$ be a locally compact space whose one-point compactification has genus 0. 
Let $n$ be a natural number. Let $P$ be the set of distinct $2n+1$ points.
For each $A  \in \basx$ let $ \nu(A) = i/n$ if $\sharp A = 2i$ or $2i+1$, where $ \sharp A$ is the number of points in $ A \cap P$.
When $X$ is compact, a set function defined in this way is a solid-set function 
(see Example 2.1 in \cite{Aarnes:Pure}, Examples 4.14 and 4.15 in \cite{QfunctionsEtm}). 
By Theorem \ref{tmXtoXha}  $\nu$ is a solid-set function on $X$; it extends to a unique topological measure on $X$ 
that assumes values $0, 1/n, \ldots, 1$. 
\end{example}

We conclude with an example of another  infinite topological measure.

\begin{example} \label{mojexLC}
Let $X=\r^n$ for any $n \ge 2$, and $\la$ be the Lebesque measure on $X$.
For $U \in \bosx$ define $\mu(U) =0$ if $0 \le \la(U) \le 1$ and $\mu(U) = \la(U)$ if
$\la(U) >1$. For $C \in \bcsx$ define $\mu(C) = 0$ if $0 \le \la(C) < 1$ and 
$\mu(C) =\la(C)$ if $\la(C) \ge 1$.  It is not hard to check the first three conditions 
of Definition \ref{DeSSFLC}. From Remark \ref{easyRn} it follows that $\mu$ is 
a solid set function on $X$. 
By Theorem \ref{extThLC} $\mu$ extends uniquely to a topological measure on $X$, 
which we also call $\mu$.  Note that $\mu(X) = \infty$. $\mu$ is not subadditive, for  
we may cover a compact ball with Lebesque measure greater than 1 by finitely many balls 
of Lebesque measure less than 1. Hence, $\mu$ is not a measure.
\end{example}

{\bf{Acknowledgments}}:
This  work was conducted at the Department of Mathematics at the University of Illinois at Urbana-Champaign and 
the Department of Mathematics at the University of California Santa Barbara. The author would like to thank both departments
for their hospitality and supportive environments.

${   }$ \\
\noindent
Department of Mathematics \\
University of California Santa Barbara \\
552 University Rd.,  \\
Isla Vista, CA 93117, USA \\
e-mail: svbutler@ucsb.edu  \\ 

\end{document}